%% file: LieDyn_formal_21avril2016.tex
\documentclass[11pt,reqno,oneside,a4paper]{amsart}

\input{macroQMC.tex}

\usepackage[hyperindex=true, colorlinks=false]{hyperref}
\usepackage[final]{showkeys} 

\begin{document}
\title[]{Normalization in Lie algebras via mould calculus\\[1ex] and applications}

\author{Thierry Paul}
\address{
CMLS, Ecole
polytechnique, CNRS, Universit\'e Paris-Saclay,
91128 Palaiseau Cedex,
France}
\email{thierry.paul@polytechnique.edu}


\author{David Sauzin}
\address{CNRS UMI 8028--IMCCE \\
77 av.\ Denfert-Rochereau, 75014 Paris, France\\
\newline CNRS UMI 3483--Laboratorio Fibonacci \\
CRM E.\ De Giorgi, Scuola Normale Superiore di Pisa\\
Piazza dei Cavalieri 3, 56126 Pisa, Italy}
\email{David.Sauzin@obspm.fr}


\date{\today}
\maketitle

\vspace{-1.1ex}

\begin{abstract}
  We establish \'Ecalle's mould calculus in an abstract Lie-theoretic
  setting and use it to solve a normalization problem, which covers
  several formal normal form problems in the theory of dynamical systems.
  The mould formalism allows us to reduce the Lie-theoretic problem to
  a mould equation, the solutions of which are remarkably explicit and
  can be fully described by means of a gauge transformation group.

  The dynamical applications include 
  the construction of Poincar\'e-Dulac formal normal forms for a
  vector field around an equilibrium point, 
  a formal infinite-order multiphase averaging procedure for vector
  fields with fast angular variables (Hamiltonian or not), 
  or the construction of Birkhoff normal forms both in classical and
  quantum situations.
  As a by-product we obtain, in the case of harmonic oscillators, the
  convergence of the quantum Birkhoff form to the classical one,
  without any Diophantine hypothesis on the frequencies of the
  unperturbed Hamiltonians.
\end{abstract}

\vspace{-1.1ex}

\tableofcontents

\section*{Introduction}

%
We are interested in the following situation:
given $X_0, B \in \cL$,
where $\cL$ is a Lie algebra over a field~$\kk$ of characteristic zero,
%
%
we look for a Lie algebra automorphism~$\Psi$ which maps $X_0+B$ to an
element of~$\cL$ which commutes with~$X_0$.
We call such a~$\Psi$ a ``normalizing automorphism'' and $\Psi(X_0+B)$
is then called a ``normal form'' of $X_0+B$.
Our key assumption will be that~$B$ can be decomposed into a sum $B =
\sum B_n$ of eigenvectors of the inner derivation
%
%
$\ad_{X_0} \col Y \mapsto [X_0,Y]$.
We will also assume that~$\cL$ is a ``complete filtered Lie algebra''
(Definition~\ref{DefCFLA} below),
which will allow us to look for~$\Psi$ in the form of the exponential
of an auxiliary inner derivation.

Our first aim in this article is to introduce \'Ecalle's ``mould calculus''
for this situation, in the simplest possible way, and to use it to
find an explicit solution to the normalization problem:
we will obtain $\Psi = \exp(\ad_Y)$ and $\Psi(X_0+B)=X_0+Z$ with
$Y,Z\in\cL$ given by explicit formal series involving all possible iterated
Lie brackets $[B_{n_r},[\ldots[B_{n_2},B_{n_1}]\ldots]]$.
It is the family of coefficients that one puts in front of these
iterated Lie brackets that is called a ``mould'';
we shall be led to an equation for the moulds associated with~$Y$
and~$Z$, and our second main result will consist in describing all
its solutions, especially all those which are ``alternal moulds'' (see
below),
and giving an algorithm to compute them.
%

\setlength{\tabcolsep}{1pt}
\begin{table}
\centering
\resizebox{\columnwidth}{!}{%
\begin{tabular}{ | M{2.3cm} | M{5.5cm} | M{4.7cm} | M{4.65cm} |  }
\hline
&
$\begin{array}{c}
\text{Nature of the Lie algebra~$\cL$}\\ 
\text{and its Lie bracket}
\end{array}$
&
  Element to be normalized
   $X = X_0 + B$, \, $B=\sum B_n$
&
   Normalization
   $\ex^{\ad_Y} X = X_0 + Z$
\\[1ex] \hline
  Poincar\'e-Dulac normal form
& 
  $\begin{array}{c}
  \text{Formal vector fields} \\
   \text{in $z_1,\ldots,z_N$ with}\\
   \text{their natural Lie bracket}
   \end{array}$  
& 
  $\begin{array}{c}
     X_0= \sum\limits_{j=1}^N \om_jz_j\pa_{z_j}
\\[2ex]
     B = \sum\limits_{n\in\cN} B_n  \\[2ex]
      \cN = \{\, \scal{k}{\om} - \om_j\,\}\subset\C
   \end{array}$  
& 
  $\begin{array}{c}
     \ex^{\ad_Y} X = \Phi\ii_* X, \\[1ex]
     \text{$\Phi \defeq$ formal time-$1$}\\
     \text{map for~$Y$,}\\[.5ex]
     \text{$Z$ resonant}
   \end{array}$
\\[7ex]  \hline
  Birkhoff normal form
&  
  $\begin{array}{c} 
\text{Formal Hamiltonians} \\ 
\text{in $x_1,y_1,\dots,x_d,y_d$} \\
\text{with Poisson bracket} \\
\text{for $\sum \dd x_j\wedge \dd y_j$}
\end{array}$
&
  $\begin{array}{c}
     X_0=\sum\limits_{j=1}^d \dem \om_j(x_j^2+y_j^2) \\[2ex]
     B=\sum\limits_{n\in\Z^d} B_{n}\\[2ex] 
\lam(n) = \I\, \scal{n}{\om}
   \end{array}$
&
  $\begin{array}{c} 
\ex^{\ad_Y} X = X\circ\Phi, \\
     \text{$\Phi \defeq$ formal time-$1$ map}\\
     \text{for the Hamiltonian}\\
\text{vector field $\{Y,\cdot\,\}$,}\\
     \text{$Z$ resonant}
 \end{array}$
\\[6ex]  \hline
  Multiphase averaging
&
  $\begin{array}{c} 
\text{Vector fields or Hamiltonians}\\
\text{$\sum F_j \frac{\pa\,\;}{\pa\ph_j} + 
\sum G_k \frac{\pa\,\;}{\pa I_k}$
or $H(\ph,I)$}\\
      \text{trigonometric polyn.\ in~$\ph$,}\\
\text{smooth in~$I$, formal in~$\eps$}
 \end{array}$
&
  $\begin{array}{c}
    \text{$X_0 = \sum \om_j \frac{\pa\,\;}{\pa\ph_j}$\;\,or\,\;$\scal{\om}{I}$}\\[1.5ex]
     B = \sum\limits_{n\in\Z^d} B_n\\[2ex]
\lam(n) = \I \, \scal{n}{\om}
   \end{array}$
&
  $\begin{array}{c} 
     \text{$\ex^{\ad_Y} X = \Phi\ii_* X$ or $X\circ\Phi$,} \\[.5ex]
     \text{$\Phi \defeq$ formal time-$1$}\\
     \text{map for~$Y$ or $\{Y,\cdot\,\}$,}\\[.5ex]
     \text{$Z$ resonant, formal in $\eps$}
 \end{array}$
\\[1.5ex]  \hline
  Quantum perturbation theory
  &
$\begin{array}{c}
\text{$\LCE[[\eps]]$, operators in a Hilbert}\\
\text{formal in~$\eps$, finite-column}\\
 \text{w.r.t.\ an orthonormal basis~$\ebas$,}\\
  \qu{\cdot\,}\cdot \defeq \frac{1}{\I\hb}\times\text{commutator}
\end{array}$
 &
$\begin{array}{c}
     X_0 = \sum\limits_{k\in I}\ket{e_k} \, E_k \, \bra{e_k} \\[1ex]
     B = \sum\limits_{n\in\cN} B_n\\[1ex] 
   \cN =  \{\,\tfrac{1}{\I\hb}(E_\ell-E_k)\,\} \subset \C
     \end{array}$
& 
    $\begin{array}{c}
\ex^{\ad_Y} X = \ex^{\frac{1}{\I\hb}Y} X \, \ex^{-\frac{1}{\I\hb}Y}, \\[1.5ex]
\text{$Z$ block-diagonal on~$\ebas$}\\
\text{and formal in~$\eps$}
     \end{array}$ 
  \\[8ex]  \hline
  Quantum perturbation theory uniform in  $\hb\to 0$
  &
$\begin{array}{c}
\text{$\LREfb[[\eps]]$, operators in $L^2(\R^d)$}\\
\text{obtained by}\\
\text{Weyl quantization}\\[.5ex]
  \qu{\cdot\,}\cdot \defeq \frac{1}{\I\hb}\times\text{commutator}
\end{array}$
 &
$\begin{array}{c}
          X_0 = \hspace{2.9cm} \text{}\\
  \text{ } \hspace{.55cm}   -\dem\hb^2\De_{\R^d}+\sum\dem\om_j^2 x_j^2
     \\[.7ex]
    B = \sum\limits_{n \in \Z^d} B_{n}\\[.7ex] 
\lam(n) = \I \, \scal{n}{\om}
     \end{array}$
& 
    $\begin{array}{c}
\ex^{\ad_Y} X = \ex^{\frac{1}{\I\hb}Y} X \, \ex^{-\frac{1}{\I\hb}Y}, \\[1ex]
\text{symbol of $Z$ tending}\\
\text{to classical B.N.F.}\\
\text{as $\hb\to0$}
     \end{array}$ 
  \\[8ex]  \hline
\end{tabular}%
}
\medskip
\caption{Synthetic overview of applications to dynamics}
\label{tableApplyn}
\end{table}


Next, we give applications of our result to perturbation theory in classical and
quantum dynamics.
Indeed, there are several formal normalization problems for dynamical systems or
quantum systems which can be put in the above form:
\begin{enumerate}[--]
\item
  the construction of Poincar\'e-Dulac formal normal forms for a
  vector field around an equilibrium point with diagonalizable linear
  part,
  taking for~$X_0$ the linear part of the vector field and for~$\cL$
  the Lie algebra of formal vector fields;
\item
  the construction of Hamiltonian Birkhoff normal forms at an elliptic
  equilibrium point,
  taking for~$X_0$ the quadratic part of the Hamiltonian and for~$\cL$
  the Poisson algebra of formal Hamiltonian functions;
\item
  the elimination at every perturbative order (``averaging'') of a
  fast angular variable $\ph\in\T^d$ with fixed frequency $\om\in\R^d$
  in a slow-fast vector field (Hamiltonian or not),
  taking $X_0 = \sum \om_j \frac{\pa\,\;}{\pa\ph_j}$;
\item
  the construction of quantum Birkhoff normal forms in a
  Rayleigh-Schr\"odinger-type situation,
%
%
  taking for~$X_0$ the unperturbed part of the quantum Hamiltonian and
  for~$\cL$ a Lie algebra of operators of the underlying Hilbert space.
\end{enumerate}

%
There is a fifth application, dealing with the way the coefficients of
the quantum Birkhoff normal forms formally converge, as $\hb\to0$, to
those of the classical Birkhoff normal form.

The reader will find a synthetic overview of the dynamical
applications in Table~\ref{tableApplyn} on p.~\pageref{tableApplyn} and more explanations in
Sections~\ref{secPDNF}--\ref{secSemiCl}, particularly about the way
one can use ``homogeneity'' to decompose a given~$B$ into a sum
$\sum B_n$ of eigenvectors of $\ad_{X_0}$ 
(the indices~$n$ belong to a countable set
depending on the chosen example; the eigenvalue associated with~$n$ is
denoted by~$\lam(n)$ when it is not~$n$ itself).

In our view, one of the merits of the Lie-theoretic framework we have
devised is its unifying power.
Indeed, the dynamical applications we have mentioned are well-known,
but what is new is the way we obtain each of them as a by-product of
one theorem on the normalization problem in a Lie algebra which itself
derives from one theorem on the solutions of a certain mould
equation. 
The fact that one can use exactly the same moulds in all these
applications is in itself remarkable.
This point of view offers a better understanding of the combinatorics
involved in these applications.
In particular we shall see that our approach gives a more direct way
of relating quantum and classical normal forms (last line of Table
\ref{tableApplyn}).

Normal forms in completed graded Lie algebras have been studied in
\cite{menousjmp}, which is dedicated to logarithmic derivatives
associated with graded derivations, motivated by perturbative quantum field
theory. However, we see no obvious way of deducing our main results
from \cite{menousjmp}, which works in a different context and adopts a
more Hopf-algebraic point of view without involving any moulds.

A forthcoming paper \cite{Dyn} will be devoted to normal form problems
similar to the ones studied in the present article (including
applications to classical and quantum dynamics), but in the framework
of \textit{Banach scales of Lie algebras}; there, the focus will be on
more quantitative results, which can be obtained thanks to the mould
representation of the solution in a more analytic context.

\medskip

Our method relies on \'Ecalle's concept of ``alternal mould''
(\cite{Eca81}, \cite{Eca93}) and owes a lot to the
article \cite{EV95} (particularly the part on the so-called ``mould of the regal prenormal
form'' in it).
Our approach is however slightly different, and it incorporates a more
direct introduction of alternality, because we work in a Lie algebra
rather than with an associative algebra of operators which would
themselves act on an associative algebra.
We do not require from the reader any previous knowledge of the mould
formalism.
We will provide original self-contained proofs, except for a few elementary
facts of \'Ecalle's theory the proof of which can be found \eg in
\cite{mouldSN}; 
at a technical level, we shall use crucially the ``dimoulds''
introduced in \cite{mouldSN}.

The core of our work consists in finding and describing the alternal
moulds solutions to a certain equation.
This is tightly related to algebraic combinatorics. For instance,
finite-support alternal moulds can be identified with the primitive
elements of a certain combinatorial Hopf algebra, and general alternal
moulds with the infinitesimal characters of the dual Hopf
algebra. Moreover, the mould counterpart to the grouplike elements of this
Hopf algebra and the characters of its dual is embodied in \'Ecalle's
concept of ``symmetrality''.
Solving our mould equation will lead us to a generalisation of the
classical character of the combinatorial Hopf algebra
$\operatorname{QSym}$ related to the Dynkin Lie idempotent.
However, in this article, we shall not use the language of Hopf
algebras but rather stick to \'Ecalle's mould calculus and its
application to our Lie-theoretic problem.


\medskip

%
The article is divided into three parts.
\begin{enumerate}[--]
\item
The part ``Main general results'' contains two sections. 
The first is devoted to the statement of the first main result,
Theorem~\ref{thmA}, in the context of complete filtered Lie algebras.
The second section gives the minimum amount of the mould formalism necessary
to state the second main result, Theorem~\ref{thmB}, about the set of
all alternal solutions to a certain mould equation.
\item
  The part ``Lie mould calculus'' contains two sections:
  Section~\ref{secLMCpfThmA} explains the origin of the notion of
  alternal mould in relation with computations in a Lie algebra, and
  then derives the proof of Theorem~\ref{thmA} from
  Theorem~\ref{thmB}.
  Section~\ref{secpfThmB} gives the proof of Theorem~\ref{thmB} with
  the help of ``dimoulds''.
\item
The part ``Five dynamical applications'' contains five sections, each
devoted to a particular application of Theorem~\ref{thmA}:
Section~\ref{secPDNF} for Poincar\'e-Dulac normal forms of formal
vector fields,
Section~\ref{secCBNF} for classical Birkhoff normal forms of formal
Hamiltonians,
Section~\ref{secmultiphas} for the elimination of a fast angular phase
in formal slow-fast vector fields,
Section~\ref{secQBNF} for quantum Birkhoff normal forms of formal
perturbations of certain quantum Hamiltonians,
Section~\ref{secSemiCl} for the formal convergence of quantum Birkhoff
normal forms to classical Birkhoff normal forms as $\hb\to0$ for
perturbations of harmonic oscillators.
To our knowledge, the latter result, valid for arbitrary frequencies, is new and generalizes earlier ones
\cite{GP1} \cite{DGH}, which required a Diophantine condition.
These applications, though more specialized than the main general
results, are written in a self-contained way so as to be (hopefully)
accessible to readers who are not specialists of the different domains
they cover.
\end{enumerate}

\newpage


\centerline{\Large\sc Main general results}
\addcontentsline{toc}{part}{\vspace{-2.1em} \\ \sc Main general results\vspace{.4em}}



\section{Normalization in complete filtered Lie algebras
  (Theorem~\ref{thmA})}


%
Throughout the article we use the notations
\[
\N = \{0,1,2,\ldots\},
\qquad \I = \sqrt{-1}.
\]

%
\begin{definition}    \label{DefCFLA}
A ``complete filtered Lie algebra'' is a Lie algebra $\big( \cL, [.\,,.]
\big)$ together with a sequence of subspaces
\[
\cL = \cL\ugeq 0 \supset\cL\ugeq1\supset\cL\ugeq2 \supset \ldots \qquad
\text{with $[\cL\ugeq m,\cL\ugeq n] \subset \cL\ugeq{m+n}$ for all $m,n\in\N$}
\]
(exhaustive decreasing filtration compatible with the Lie bracket) such that 
$\bigcap\cL\ugeq m=\{0\}$ (the filtration is separated)
and~$\cL$ is a complete metric space for the distance
$d(X,Y) \defeq 2^{-\ord(Y-X)}$,
where we denote by $\ord \col \cL\to\N\cup\{\infty\}$ the order
function associated with the filtration
(function characterized by 
$\ord(X)\geq m \Leftrightarrow X\in\cL\ugeq m$).
\end{definition}
%

The completeness assumption will be used as follows: 
given a set~$I$, a family $(Y_i)_{i\in I}$ of~$\cL$ is said to be
``formally summable'' if,
for any $m\in\N$, the set $\{\, i\in I \mid Y_i\notin\cL\ugeq m \,\}$ is finite;
one can then check that the support of this family is countable
(if not~$I$ itself) and that, for any exhaustion $(I_k)_{k\in\N}$ of
this support by finite sets, the sequence $\sum_{i\in I_k} Y_i$ is
Cauchy, with a limit which is independent of the exhaustion---this
common limit is simply denoted by $\sum_{i\in I} Y_i$.


Here is a simple and useful example of a formally summable series of operators in~$\cL$:
for any $Y\in\cL\ugeq1$ and $r\in\N$, the operator $(\ad_Y)^r$ maps~$\cL$
in~$\cL\ugeq r$, hence, for every $X\in\cL$, the series
$\ex^{\ad_Y}(X) \defeq \sum_{r=0}^\infty \frac{1}{r!} (\ad_Y)^r(X)$ is
formally summable in~$\cL$. 
This allows us to define the operator $\ex^{\ad_Y}$, which is a Lie
algebra automorphism because $\ad_Y$ is a Lie algebra derivation. 


Our first main result is
\begin{theorem}\label{thmA}
Let~$\kk$ be a field of characteristic zero.
There exist families of coefficients 
\be   \label{eqFBinkk}
F^{\lam_1,\ldots,\lam_r},\, G^{\lam_1,\ldots,\lam_r} \in \kk
\quad
\text{for $r\ge1$, $\lam_1,\ldots,\lam_r \in \kk$,}
\ee
explicitly computable by induction on $r$, which satisfy the following:
given a complete filtered Lie algebra~$\cL$ over~$\kk$ and $X_0\in\cL$,
given a set~$\cN$ and a formally summable family $(B_n)_{n\in\cN}$ of~$\cL$
such that each~$B_n$ has order~$\ge1$ and is an eigenvector of
$\ad_{X_0}$,
one has
\be\label{eqA1}
[X_0,Z]=0 ,
\qquad
\ex^{\ad_Y} \Big(X_0+\sum_{n\in\cN} B_n \Big) =X_0+Z,
\ee
where $Z,Y\in\cL\ugeq1$ are defined as the following
sums of formally summable families:
\begin{align}
\label{eqmouldexpZ}
Z &= \sum_{r\geq 1} \;\sum_{n_1,n_2,\ldots,n_r\in\cN}\,
\frac{1}{r} F^{\lam(n_1),\lam(n_2),\ldots,\lam(n_r)} 
[B_{n_r},[\ldots[B_{n_2},B_{n_1}]\ldots]]
\\[.6ex]
\label{eqmouldexpY}
Y &= \sum_{r\geq 1} \;\sum_{n_1,n_2,\ldots,n_r\in\cN}\,
\frac{1}{r} G^{\lam(n_1),\lam(n_2),\ldots,\lam(n_r)} 
[B_{n_r},[\ldots[B_{n_2},B_{n_1}]\ldots]]
\end{align}
with
\be \label{eqdeflam}
\lam \col \cN \to \kk,
\qquad
\lam(n) \defeq \text{eigenvalue of $B_n$}.
\ee
\end{theorem}
%
The proof of Theorem~\ref{thmA} is in Section~\ref{secPfThmA}.

As we shall see, the families
$F^\bul = (F^{\lam_1,\ldots,\lam_r})$ 
and $G^\bul = (G^{\lam_1,\ldots,\lam_r})$ 
are not unique, but
$(F^\bul,G^\bul)$ is in one-to-one correspondence with an auxiliary
family called \emph{gauge generator},
which can be chosen arbitrarily among resonant
alternal moulds (see the definitions in Section~\ref{secDefMoulds}). 
We will see that, for any choice of the gauge generator, one has
$F^{\lam_1,\ldots,\lam_r}=0$ whenever $\lam_r+\dots+\lam_1\neq0$ and
\begin{multline}   \label{eqfacileF}
\lam_r(\lam_r+\lam_{r-1}) \cdots(\lam_r+\cdots+\lam_2) \neq0,
\ens \lam_r+\dots+\lam_1=0
\imp\\
F^{\lam_1,\ldots,\lam_r} = \frac{1}{\lam_r(\lam_r+\lam_{r-1}) 
\cdots(\lam_r+\cdots+\lam_2)}.
\end{multline}
The formulas are much more complicated when the denominator vanishes,
but there still is an explicit algorithm to compute every coefficient
$F^{\lam_1,\ldots,\lam_r}$ or $G^{\lam_1,\ldots,\lam_r}$ depending on the chosen
gauge generator:
see formulas \eqref{eqinducinit}--\eqref{eqGlogS} in Section~\ref{secDefMoulds}.


\begin{remark} 
We may accept~$0$ as an eigenvector, \ie some of the $B_n$'s may
vanish and $\lam(n)$ need not be specified for those values of~$n$.
Since the support of a summable family is at most countable, one can always
choose
\be   \label{eqcasecNNst}
\cN = \N^*
\ee
without loss of generality
(by numbering the support of $(B_n)$ and, if this support is finite,
setting $B_n=0$ for the extra values of~$n$).
On the other hand, one can decide to group together the
eigenvectors associated with the same eigenvalue and take for~$\cN$
the countable subset of~$k$ consisting of the eigenvalues which appear
in the problem,
in which case
\be   \label{eqcNincluskk}
\cN \subset \kk, \qquad
\lam(n)=n \ens\text{for $n\in\cN$}
\ee  
(this latter choice is the one of \cite{EV95}).
In this article we do not opt for any of these two choices and simply
consider a general eigenvalue map~\eqref{eqdeflam} with
arbitrary~$\cN$
(without assuming $B_n\neq0$ for each~$n$).
\end{remark}


\begin{remark}   \label{remexpadSB}
The factor~$\frac{1}{r}$ in~\eqref{eqmouldexpZ}--\eqref{eqmouldexpY}
is just a convenient normalization.
We shall see in Section~\ref{secPfRem} that the inner derivation $\ad_Y$ itself can be written
\begin{multline}   \label{eqadYGB}
\ad_Y = 
\sum_{r\geq 1} \;\sum_{n_1,n_2,\ldots,n_r\in\cN}\,
\frac{1}{r} G^{\lam(n_1),\lam(n_2),\ldots,\lam(n_r)} 
[\ad_{B_{n_r}},[\ldots[\ad_{B_{n_2}},\ad_{B_{n_1}}]\ldots]]
\\[1ex]
= \sum_{r\geq 1} \;\sum_{n_1,\ldots,n_r\in\cN}\,
G^{\lam(n_1),\cdots, \lam(n_r)} \ad_{B_{n_r}}\cdots \ad_{B_{n_1}}
\end{multline}
(no more factor $\frac{1}{r}$ in the last series!---note that in
general the individual composite operators
$\ad_{B_{n_r}}\cdots \ad_{B_{n_1}}$ are not derivations of~$\cL$).
We shall also define a family of coefficients~$S^\bul$ tightly related
to~$G^\bul$ such that 
\be   \label{eqexpadYSB}
\ex^{\ad_Y} = \ID +
\sum_{r\geq 1} \;\sum_{n_1,\ldots,n_r\in\cN}\,
S^{\lam(n_1),\cdots, \lam(n_r)} \ad_{B_{n_r}}\cdots \ad_{B_{n_1}}.
\ee
\end{remark}


\begin{remark}   \label{rem:othersolns}
If $Z,Y\in\cL\ugeq1$ solve equation~\eqref{eqA1}, then any $W\in\cL\ugeq1$
such that $[X_0,W]=0$ gives rise to a solution $(\ti Z,\ti Y)$ by
setting $\ti Z\defeq \ex^{\ad_W} Z$ and
$\ti Y \defeq \BCH(W,Y) = W+Y+\dem[W,Y]+\cdots$,
the Baker-Campbell-Hausdorff series, which is formally summable and satisfies
$\ex^{\ad_{\ti Y}} = \ex^{\ad_W} \ex^{\ad_Y}$.

In Section~\ref{secCBNF}, we shall see an example in which~$Z$ is
unique but~$Y$ is not.
\end{remark}


We conclude this section with a ``truncated version'' of Theorem~\ref{thmA}:
\begin{addend}
Take $\cL$, $X_0$, $(B_n)_{n\in\cN}$ and $\lam\col\cN\to\kk$ as in the
assumptions of Theorem~\ref{thmA}.
Then, for each $m\in\N^*$,
the set $\cN_m \defeq \{\, n\in\cN \mid B_n\notin\cL\ugeq m\}$ is
finite and the finite sums
\begin{align}
\label{eq:defZm}
\tr{Z}{m} &\defeq \sum_{r= 1}^{m-1} \,\sum_{n_1,\ldots,n_r\in\cN_m}
\frac{1}{r} F^{\lam(n_1),\ldots,\lam(n_r)} 
[B_{n_r},[\ldots[B_{n_2},B_{n_1}]\ldots]],
\\[.6ex]
\label{eq:defYm}
\tr{Y}{m} &\defeq\sum_{r=1}^{m-1} \, \sum_{n_1,\ldots,n_r\in\cN_m}
\frac{1}{r} G^{\lam(n_1),\ldots,\lam(n_r)} 
[B_{n_r},[\ldots[B_{n_2},B_{n_1}]\ldots]]
\end{align}
define $\tr{Z}{m}, \tr{Y}{m}\in\cL\ugeq1$ satisfying
$[X_0,\tr{Z}{m}]=0$ and
\be\label{eqA2}
\ex^{\ad_{\tr{Y}{m}}} \Big(X_0+\sum_{n\in\cN}B_n\Big)
=X_0 + \tr{Z}{m} \mod \cL\ugeq m.
\ee
\end{addend}

The proof is in Section~\ref{secPfAddend}.

\section{The mould equation and its solutions
  (Theorem~\ref{thmB})}   \label{secDefMoulds}


%
We now describe the part of \'Ecalle's mould formalism which will allow us to
construct the aforementioned families of coefficients. 
This will lead us to an equation, of which we will describe
all solutions.

\parage
Let~$\kk$ a field and $\cN$ a nonempty set, considered as an
alphabet. We denote by~$\UN$ the corresponding free monoid, whose
elements are called \emph{words},
\[
\UN \defeq \{\, \un = n_1\cdots n_r \mid 
r\in\N,\; n_1,\ldots,n_r\in\cN \,\}.
\]
The monoid law is word \emph{concatenation}: 
$\ua \, \ub = a_1\cdots a_r b_1 \cdots b_s$ for $\ua = a_1\cdots a_r$ and
$\ub = b_1 \cdots b_s$.
Its unit is the empty word, denoted by~$\est$, the only word of length~$0$.
The length of a word~$\un$ is denoted by $r(\un)$.
(Given $r\in\N$, we sometimes identify the set of all words of
length~$r$ with~$\cN^r$.)


We call \emph{mould} any map $\UN \to \kk$.
It is customary to denote the value of the mould on a word~$\un$ by
affixing~$\un$ as an upper index to the symbol representing the mould,
and to refer to the mould itself by using a big dot as upper index;
hence $M^\bul$ is the mould, the value of which at~$\un$ is denoted by
$M^\un$.

For example, the families of coefficients $F^\bul,G^\bul$ referred to
in Theorem~\ref{thmA} can be considered as moulds, taking $\cN=\kk$ as
alphabet. For that reason, from now on, we will write
$F^{\lam_1\cdots\lam_r}$ and $G^{\lam_1\cdots\lam_r}$ to denote the
individual coefficients rather than
$F^{\lam_1,\ldots,\lam_r}$ or $G^{\lam_1,\ldots,\lam_r}$
as in~\eqref{eqFBinkk}.

The set $\kk^\UN$ of all moulds is clearly a linear space over~$\kk$. It is
also an associative $\kk$-algebra (usually not commutative): 
\emph{mould multiplication} is induced by word concatenation,
\be   \label{eqdefmouldmult}
P^\bul = M^\bul \times N^\bul 
\ens \text{is defined by} \ens 
\un \in \UN \mapsto P^\un \defeq 
\sum_{\un = \ua \, \ub} M^\ua N^\ub
\ee
(summation over all pairs of words $(\ua,\ub)$ such that 
$\un = \ua \, \ub$, including $(\un,\est)$ and $(\est,\un)$, thus
there are $r(\un)+1$ terms in the sum).\footnote{\label{ftn:dualkUN}
  The linear space $\kk^\UN$ can be identified with the dual of
  $\kUN$, the $\kk$-vector space consisting of all linear combinations
  of words (formal sums of the form $\sum x_\un\, \un$, with finitely
  many nonzero coefficients $x_\un\in\kk$): 
  the mould $M^\bul$ gives rise to the linear form
  $x\in\kUN\mapsto M^\bul(x)\in\kk$ defined by
  $M^\bul(\sum x_\un\, \un) = \sum x_\un M^\un$.
  The associative algebra structure on $\kk^\UN$ is then dual to the
  coalgebra structure induced on $\kUN$ by ``word deconcatenation'',
  for which the coproduct is
  $\De(\un) = \sum\limits_{\un=\ua\,\ub} \ua \otimes\ub$.
}
The multiplication unit is the elementary mould~$1^\bul$ defined by
$1^\est = 1$ and $1^\un = 0$ for $\un\neq\est$.
It is easy to see that a mould~$M^\bul$ is invertible if and only if
$M^\est\neq0$;
we then denote its multiplicative inverse by~$\inv M^\bul$.

The Lie algebra associated with the associative algebra $\kk^\UN$ will be
denoted $\LIE(\kk^\UN)$ (same underlying vector space, with bracketing
$[M^\bul,N^\bul] \defeq M^\bul\times N^\bul - N^\bul\times M^\bul$).


The order function $\ord \col \kk^\UN \to \N\cup\{\infty\}$ defined by 
\be   \label{eqdefordmould}
\ord(M^\bul)\geq m \quad\Leftrightarrow \quad
\text{ $M^\un=0$ whenever $r(\un)<m$ }
\ee
allows us to view~$\kk^{\UN}$ as a complete filtered associative
algebra (because the distance $d(M^\bul,N^\bul) \defeq 2^{-\ord(N^\bul-M^\bul)}$ makes it a complete
metric space and
$\ord(M^\bul\times N^\bul) \ge \ord(M^\bul)+\ord(N^\bul)$).
We can thus define the mutually inverse exponential and logarithm maps by the following summable series:
\[
M^\est = 0 \imp
\ex^{M^\bul} \defeq 1^\bul + 
\sum_{k\ge1} \tfrac{1}{k!} (M^\bul)^{\times k},
\quad
\log(1^\bul+M^\bul) \defeq \sum_{k\ge1} \tfrac{(-1)^{k-1}}{k} (M^\bul)^{\times k}.
\]


\parage   \label{secalternalmoulds}
\'Ecalle's notion of ``alternality'' is of fundamental importance.
Its motivation will be made clear in Section~\ref{secLieMouldExp}.
The idea is that, since in the situation of Theorem~\ref{thmA} we will
use a mould~$M^\bul$ as a family of coefficients to be multiplied by
iterated Lie brackets (as~$F^\bul$ in \eqref{eqmouldexpZ} or~$G^\bul$
in~\eqref{eqmouldexpY}), it is natural to impose some symmetry (or,
rather, antisymmetry) on the coefficients so as to take into account the
antisymmetry of the Lie bracket.
For instance, the sum over all two-letter words contains expressions
like
$\dem M^{n_1 n_2}[B_{n_2},B_{n_1}] + \dem M^{n_2 n_1}[B_{n_1},B_{n_2}]$,
which coincide with 
$\dem (M^{n_1 n_2}-M^{n_2 n_1})[B_{n_2},B_{n_1}]$,
%
%
so it is natural to impose
\be   \label{eqMaltlengthtwo}
M^{n_1 n_2} + M^{n_2 n_1} = 0
\quad\text{for all $n_1,n_2\in\cN$,}
\ee
so as to reduce to~$1$ the number of degrees of freedom associated
with the words~$n_1n_2$ and~$n_2n_1$.
Alternality is a generalisation of~\eqref{eqMaltlengthtwo} for all lengths~$\ge2$.


The definition of alternality is based on word shuffling.
Roughly speaking, the shuffling of two words~$\ua$ and~$\ub$ is the
set\footnote{or rather the sum---see footnote~\ref{ftn:shuff}}
of all words obtained by interdigitating the letters of~$\ua$
and~$\ub$ while preserving their internal order in~$\ua$ or~$\ub$;
the number of different ways a word~$\un$ can be obtained out of~$\ua$
and~$\ub$ is called shuffling coefficient.
We make this more precise by using permutations as follows.
For $r\in\N$, we let~$\fS_r$ (the symmetric group of degree~$r$)
act to the right on the set~$\cN^r$ of all words of length~$r$ by
\be   \label{eqrightactionfSr}
\un = n_1 \cdots n_r \mapsto \un^\tau \defeq n_{\tau(1)}\cdots
n_{\tau(r)}
\quad\text{for $\tau\in\fS_r$ and $\un\in\cN^r$.}
\ee
For $0\le \ell \le r$, we set
\[
\un^\tau\uleq\ell \defeq n_{\tau(1)} \cdots n_{\tau(\ell)},
\qquad
\un^\tau\usup\ell \defeq n_{\tau(\ell+1)} \cdots n_{\tau(r)}.
\]
%
%
We also define
\[
\fS_r(\ell) \defeq \{\, \tau \in \fS_r \mid
\tau(1) < \cdots < \tau(\ell) \ens\text{and}\ens
\tau(\ell+1) < \cdots < \tau(r) \,\},
\]
with the conventions $\fS_r(0) = \fS_r(r) = \{\id\}$.

\begin{definition}   \label{def:shufflcoeff}
Given $\un,\ua,\ub\in\UN$, the ``shuffling coefficient'' of~$\un$ in
$(\ua,\ub)$ is defined to be
\be   \label{eqdefshabn}
\shabn \defeq \card\{\, 
\tau \in \fS_r(\ell) \mid
\un^\tau\uleq\ell = \ua \ens\text{and}\ens \un^\tau\usup\ell = \ub
\,\}, \quad \text{where}\ens \ell\defeq r(\ua).
\ee
\end{definition}

For instance, if $n,m,p,q$ are four distinct elements of~$\cN$,
\[
\SH{nmp}{mq}{nmqpm} = 0, \qquad
\SH{nmp}{mq}{nmmqp} = 2, \qquad
\SH{nmp}{mq}{mnqmp} = 1.
\]


%
\begin{definition}    \lable{defalt}
A mould~$M^\bul$ is said to be ``alternal'' if
$M^\est=0$  and
\be   \label{eqdefMaltshabn}
\sum_{\un\in\UN} \shabn M^\un = 0 
\quad\text{for any two nonempty words $\ua, \ub$.}
\ee
\end{definition}
%

For instance, \eqref{eqdefMaltshabn} with $\ua = n_1$ and $\ub = n_2$
yields~\eqref{eqMaltlengthtwo} and, with $\ua = n_1$ and $\ub = n_1
n_2$, it yields
\[
2 M^{n_1 n_1 n_2} + M^{n_1 n_2 n_1} = 0.
\]
Notice that any mould whose support is contained in the set of
one-letter words is alternal; so is, in particular, the elementary
mould~$I^\bul$ defined by
\be   \label{eqdefmouldI}
I^\un \defeq \IND{r(\un)=1}
\qquad\text{for any word $\un$.}
\ee


We denote by $\Alt(\cN)$ the set of alternal moulds, which is clearly
a linear subspace of $\kk^\UN$; in fact,\footnote{\label{ftn:shuff}
Word shuffling gives rise to the ``shuffling product'', 
defined by
$\ua \shuff \ub \defeq 
\sum\limits_{\tau\in\fS_r(\ell)} (\ua\,\ub)^{\tau\ii}
= \sum \shabn\, \un \in \kUN$
for a pair of words such that $r(\ua)=\ell$ and $r(\ua\,\ub)=r$ and
extended to $\kUN \times \kUN$ by bilinearity,
which makes the space $\kUN$ of footnote~\ref{ftn:dualkUN} a commutative
associative algebra. 
Alternal moulds can then be identified with the infinitesimal
characters of the associative algebra $(\kUN,\shuff)$, \ie when viewed
as linear forms they are characterized by
$M^\bul(x\shuff y) = M^\bul(x) 1^\bul(y)  + 1^\bul(x) M^\bul(y)$.
In that point of view, $\Alt(\cN)$ is a Lie subalgebra of
$\LIE(\kk^\UN)$ because $\kUN$ is a bialgebra
(\ie there is some kind of compatibility between the deconcatenation
coproduct and the shuffling product---in fact, $\kUN$ is even a Hopf algebra).
}
\[ \text{\emph{$\Alt(\cN)$ is a Lie subalgebra of $\LIE(\kk^\UN)$}} \]
(see \eg \cite[Prop.~5.1]{mouldSN}); this will play a role when
returning to the situation of Theorem~\ref{thmA}.


\parage
Given a function $\ph \col \cN \to \kk$, we denote by the same
symbol~$\ph$ its extension to~$\UN$ as a monoid morphism:
$\ph(\est)=0$ and
\be   \label{eqextendph}
\un = n_1\cdots n_r \in \UN \mapsto 
\ph(\un) = \ph(n_1) + \cdots + \ph(n_r) \in \kk
\qquad\text{if $r\ge1$.}
\ee
The formula 
\be   \label{eqdefnaph}
\na_\ph \col M^\bul \in \kk^\UN \mapsto N^\bul \in \kk^\UN, \qquad
N^\un \defeq \ph(\un) M^\un \quad\text{for any word~$\un$}
\ee
then defines a derivation of the associative algebra
$\kk^\UN$ (the Leibniz rule for $\na_\ph$ is an obvious consequence of the identity $\ph(\ua\,\ub) = \ph(\ua)+\ph(\ub)$).
For example, associated with the constant function $\ph(n)\equiv 1$ is
the derivation $\na_1$, which spells
\[
\na_1 M^\un = r(\un)M^\un
\qquad\text{for any $M\in\kk^\UN$ and $\un\in\UN$.}
\]

In the situation of Theorem~\ref{thmA}, the derivation~$\na_\lam$ associated with the
map~\eqref{eqdeflam} will play a pre-eminent role. We shall need the following

\begin{definition}   \label{deflamrespart}
Given a map $\lam \col \cN \to \kk$, we call
``$\lam$-resonant'' any mould~$M^\bul$ such that $\na_\lam M^\bul=0$
and use the notation
\[
\Alt\RES(\cN) \defeq \{\, M^\bul \in \Alt(\cN) \mid \na_\lam M^\bul =
0 \,\}.
\]
The ``$\lam$-resonant part'' of a mould~$M^\bul$ is denoted
by~$M^\bul\RES$ and defined by the formula
\[
M^\un\RES \defeq \IND{\lam(\un) = 0} \, M^\un
\qquad\text{for any word $\un$.}
\]
The ``gauge generator'' of an alternal mould~$M^\bul$ is defined as
\[
\JJ(M^\bul) \defeq \lres{ \ex^{-M^\bul} \times \na_1\big( \ex^{M^\bul} \big) }.
\]
\end{definition}


Note that the space $\Alt\RES(\cN)$ of all $\lam$-resonant alternal
moulds is a Lie subalgebra of $\Alt(\cN)$ (being the kernel of a derivation).
Clearly, the $\lam$-resonant part of a mould is $\lam$-resonant; a
mould~$M^\bul$ is $\lam$-resonant if and only if $M^\bul = M^\bul\RES$
or, equivalently, if and only if $M^\un=0$ whenever $\lam(\un)\neq0$.
We shall see later that the gauge generator of an alternal mould is
always alternal and, in fact, $\Alt\RES(\cN)$ coincides with the set
of all gauge generators of alternal moulds.

It is worth singling out the particular case of an alphabet contained
in~$\kk$:
\begin{definition}   \label{def:incluscase}
If $\cN\subset \kk$ and $\lam\col\cN\to\kk$ is the inclusion map, then
we use the word ``resonant'' instead of $\lam$-resonant, and we use the
notations
$\na$,
$\Alt_0(\cN)$,
$M^\bul_0$
and $\JJz(M^\bul)$
instead of~$\na_\lam$, 
$\Alt\RES(\cN)$,
$M^\bul\RES$,
and $\JJ(M^\bul)$.
\end{definition}



\parage
We are now in a position to state our second main result, describing all the
solutions to a certain mould equation, equation~\eqref{eqmouldFG}
below.
This result, while being of interest in itself, will yield the
main step in the proof of Theorem~\ref{thmA}.
Recall that~$I^\bul$ is the alternal mould defined by~\eqref{eqdefmouldI}.


\begin{theorem}\label{thmB}
Let $\cN$ be a nonempty set, $\kk$ a field of characteristic zero,
and $\lam \col \cN \to \kk$ a map.


\noindent (i) 
For every $A^\bul \in \Alt\RES(\cN)$,
there exists a unique pair $(F^\bul, G^\bul)$ of alternal moulds such
that
\begin{gather}
\label{eqmouldFG}
\na_\lam F^\bul = 0, \qquad
\na_\lam\big(\ex^{G^\bul}\big) = I^\bul \times \ex^{G^\bul} - \ex^{G^\bul} \times F^\bul,
\\[1ex]
\label{eqjaugeGA}
\JJ(G^\bul) =A^\bul.
\end{gather}


\noindent
(ii) Suppose that $(F^\bul, G^\bul) \in \Alt(\cN)\times\Alt(\cN)$ is a
solution to equation~\eqref{eqmouldFG}.
Then the formula
\be    \label{eqgenergaugetrsf}
J^\bul \mapsto (\ti F^\bul, \ti G^\bul) = 
\Big( \ex^{-J^\bul} \times F^\bul \times \ex^{J^\bul},
\, \log\big( \ex^{G^\bul}\times \ex^{J^\bul} \big) \Big)
\ee
establishes a one-to-one correspondence between $\Alt\RES(\cN)$
%
%
and the set of all solutions $(\ti F^\bul, \ti G^\bul) \in \Alt(\cN)\times\Alt(\cN)$
of equation~\eqref{eqmouldFG}.
Moreover,
\be   \label{eqnewgauge}
\JJ(\ti G^\bul) = \ex^{-J^\bul} \times \JJ(G^\bul) \times \ex^{J^\bul}
+ \ex^{-J^\bul} \times \na_1 \big(\ex^{J^\bul}\big).
\ee
\end{theorem} 


The proof of Theorem~\ref{thmB} is given in
Section~\ref{secpfThmB}. 
It is constructive in the sense that we will obtain the following
simple algorithm to compute the values of~$F^\bul$ and
$S^\bul \defeq \ex^{G^\bul}$ on any word~$\un$ by induction on its
length $r(\un)$:
introducing an auxiliary alternal mould~$N^\bul$,
%
%
one must take 
\be   \label{eqinducinit}
S^\est=1, \qquad F^\est = N^\est = 0
\ee
and, for $r(\un)\ge1$,
\begin{align}
\label{eqinducNR}
\lam(\un) &\neq0 \imp &
F^\un &= 0, \quad 
S^\un = \frac{1}{\lam(\un)} \Big( S^{`\un} -
\sst\sum_{\un=\ua\,\ub}
S^\ua \, F^\ub \Big), \quad 
N^\un = r(\un)\,S^\un - \sst\sum_{\un=\ua\,\ub}
S^\ua \, N^\ub, \\[1.5ex]
\label{eqinducRES}
\lam(\un)&=0 \imp &
F^\un &= S^{`\un} -
\sst\sum_{\un=\ua\,\ub}
S^\ua \, F^\ub, \quad 
S^\un = \frac{1}{r(\un)} \Big( A^\un + 
\sst\sum_{\un=\ua\,\ub}
S^\ua \, N^\ub \Big), \quad 
N^\un = A^\un,
\end{align}
where we have used the notation
$`\un \defeq n_2\cdots n_r$ for $\un = n_1 n_2\cdots n_r$
and the symbol $\sst\sum$ indicates summation over non-trivial decompositions
(\ie $\ua,\ub\neq\est$ in the above sums);
we will see that the mould~$F^\bul$ thus inductively defined is alternal
and that
\be   \label{eqGlogS}
G^\est = 0, \qquad
G^\un = \sum_{k=1}^{r(\un)} 
\frac{(-1)^{k-1}}{k}
\, \sst\sum_{\un=\ua^1\cdots\ua^k} \,
S^{\ua^1}\cdots S^{\ua^k}
\quad\text{for $\un\neq\est$}
\ee
then defines the alternal mould~$G^\bul$ which solves~\eqref{eqmouldFG}--\eqref{eqjaugeGA}.


\parage
A few remarks are in order.


\paraga    \label{paragaphcNcM}
Given alphabets~$\cM$ and~$\cN$, any map $\ph \col \cN \to \cM$ induces a map
$\ph^* \col M^\bul \in \kk^\UM \mapsto M_\ph^\bul \in \kk^\UN$ defined by
$M_\ph^{n_1\cdots n_r} \defeq M^{\ph(n_1)\cdots \ph(n_r)}$,
which is a morphism of associative algebras, mapping $\Alt(\cM)$ to
$\Alt(\cN)$ and satisfying $\na_{\mu\circ\ph}\circ\ph^* = \na_\mu$ for
any $\mu\col\cM\to\kk$.
Let $\lam \defeq \mu\circ\ph$; one can easily check that, if
$A^\bul \in \Alt\RRES(\cM)$, then the unique solution
$(F^\bul,G^\bul)\in\Alt(\cM)\times\Alt(\cM)$ of
\[
\na_\mu F^\bul = 0, \qquad
\na_\mu\big(\ex^{G^\bul}\big) = I^\bul \times \ex^{G^\bul} - \ex^{G^\bul} \times F^\bul
\]
such that $\JJm(G^\bul)=A^\bul$ is mapped by~$\ph^*$ to the unique
solution in $\Alt(\cN)\times\Alt(\cN)$ of~\eqref{eqmouldFG} with gauge
generator~$\ph^*(A^\bul)$.


\paraga    \label{paragaCanCase}
Let us call ``canonical case'' the case when $\cN=\kk$ and $\lam =$ the
identity map.
We shall see in Section~\ref{secPfThmA} that the moulds
$F^\bul, G^\bul \in \kkukk$ 
which are referred to in Theorem~\ref{thmA} and give rise to 
solutions $(Z,Y)$ of equation~\eqref{eqA1}
are the ones given by Theorem~\ref{thmB} in the canonical case with
arbitrary $A^\bul \in \Alt_0(\kk)$.
The mould~$S^\bul$ referred to in Remark~\ref{remexpadSB} is then $\ex^{G^\bul}$.

We shall see that, with the notations of Theorem~\ref{thmB}(ii),
any $J^\bul\in\Alt_0(\kk)$ gives rise to $W\in\cL\ugeq1$ such that
$[X_0,W]=0$ and the solution $(\ti Z,\ti Y)$ of~\eqref{eqA1} 
associated with $(\ti F^\bul, \ti G^\bul)$
is given by
$\ti Z = \ex^{\ad_W} Z$ and
$\ti Y = \BCH(W,Y)$,
in line with Remark~\ref{rem:othersolns}.


\paraga
In part~(i) of the statement, one may choose $A^\bul = 0$; this yields
for $(F^\bul,G^\bul)$ what we call the ``zero gauge solution of equation~\eqref{eqmouldFG}''.
In the canonical case, the zero gauge solution corresponds to what is
treated in \cite{EV95} under the name ``royal prenormal form''.
The rest of the statement and the whole proof given in Section~\ref{secpfThmB} are new.

As a consequence of the remark in Section~\ref{paragaphcNcM}, the
zero gauge solution in the general case $\lam \col \cN \to \kk$ is
obtained from the zero gauge solution in the canonical case by applying~$\lam^*$.


\paraga
Another possible normalization aimed at singling out a specific
solution of~\eqref{eqmouldFG} in $\Alt(\cN)\times\Alt(\cN)$ consists
in requiring $G^\bul\RES = 0$ (instead of requiring $\JJ(G^\bul)=0$).
There is a unique such solution and here is how one can see it.

According to the Baker-Campbell-Hausdorff formula, for arbitrary
$G^\bul,J^\bul \in \big( \kk^\UN \big)\ugeq1$ (\ie such that $G^\est=J^\est=0$),
we can write
\[
\log(\ex^{G^\bul}\times \ex^{J^\bul}) = G^\bul + J^\bul +
\cF(G^\bul,J^\bul),
\qquad
\cF(G^\bul,J^\bul) = \dem [G^\bul,J^\bul] + \cdots \in \big( \kk^\UN \big)\ugeq2,
\]
where the functional~$\cF$ satisfies
$\ord\big( \cF(G^\bul,\ti J^\bul)-\cF(G^\bul,J^\bul) \big) 
\ge \ord(\ti J^\bul-J^\bul) + 1$
for all $\ti J^\bul \in \big( \kk^\UN \big)\ugeq1$
(which is a contraction property for the distance mentioned right after~\eqref{eqdefordmould})
and preserves alternality.
Now, given a solution $(F^\bul, G^\bul) \in \Alt(\cN)\times\Alt(\cN)$
to equation~\eqref{eqmouldFG}, in view of part~(ii) of
Theorem~\ref{thmB}, we see that finding a 
solution $(\ti F^\bul,\ti G^\bul) \in \Alt(\cN)\times\Alt(\cN)$
of~\eqref{eqmouldFG} such that $\ti G^\bul\RES=0$ 
is equivalent to finding $J^\bul \in \Alt\RES(\cN)$ such that
\be   \label{eq:fixedptcF}
J^\bul = - G^\bul\RES - \big[ \cF(G^\bul,J^\bul) \big]\RES.
\ee
The fixed point equation~\eqref{eq:fixedptcF} has a unique
solution~$J^\bul$ in $\big( \kk^\UN \big)\ugeq1$ (because of the
contraction property), which is clearly $\lam$-resonant, and also
alternal (because~$\cF$ preserves alternality).
The uniqueness of the mould~$J^\bul$ entails that the solution $(\ti
F,\ti G)$ is unique (it does not depend on the auxiliary solution
$(F^\bul,G^\bul)$ we started with).


\paraga   \label{paragaSymMouldEq}
  ``Symmetral'' moulds can be defined as the elements of
\be    \label{defSymexpAlt}
\Sym(\cN) \defeq \{ \ex^{M^\bul} \mid M^\bul \in \Alt(\cN) \}
\ee
and $\big( \Sym(\cN), \times \big)$ is a group, in general
non-commutative
(see \eg \cite[Prop.~5.1]{mouldSN}; see also Remark~\ref{remdefsym} below).

Thus, using the change of unknown $S^\bul = \ex^{G^\bul}$, it is
equivalent to look for a solution $(F^\bul, G^\bul) \in \Alt(\cN)\times\Alt(\cN)$
of equation~\eqref{eqmouldFG} 
or for a solution $(F^\bul, S^\bul) \in \Alt(\cN)\times\Sym(\cN)$ of the
equation
\be    \label{eqmouldFS}
\na_\lam F^\bul = 0, \qquad
\na_\lam S^\bul = I^\bul \times S^\bul - S^\bul \times F^\bul,
\ee
and the gauge generator will then be
\be    \label{eqjaugeS}
\JJ(\log S^\bul) = \lres{ \inv S^\bul \times \na_1 S^\bul }.
\ee
This mould $S^\bul = \ex^{G^\bul}$ is the one which appears in the
algorithm \eqref{eqinducinit}--\eqref{eqinducRES};
there, $N^\bul$ is the auxiliary mould $N^\bul = \inv S^\bul \times
\na_1 S^\bul$.


\paraga
For any choice of $A^\bul \in \Alt\RES(\cN)$, from
\eqref{eqinducNR}--\eqref{eqinducRES}, one easily gets
\begin{multline}    \label{eqfacileS}
\lam(n_1n_2\cdots n_r)\lam(n_2\cdots n_r)\cdots \lam(n_r) \neq0
\imp\\
F^{n_i\cdots n_r} = 0 \ens\text{for $i = 1,\ldots, r$}
\ens\text{and}\ens
S^{n_1\cdots n_r} = \frac{1}{
\lam(n_1n_2\cdots n_r)\lam(n_2\cdots n_r)\cdots \lam(n_r)
},
\end{multline}
%
%
whence~\eqref{eqfacileF} follows by~\eqref{eqinducRES} and Section~\ref{paragaCanCase}.

Note that it may happen that $\lam(\un)\neq0$ for every nonempty
word~$\un$,
in which case $\Alt\RES(\cN) = \{0\}$ and there is only one solution
$(F^\bul,G^\bul)\in \Alt(\cN)\times\Alt(\cN)$ to
equation~\eqref{eqmouldFG}, namely $F^\bul=0$ 
and $G^\bul =$ logarithm of the mould~$S^\bul$ defined by~\eqref{eqfacileS}.

For instance, this is what happens if $\cN=\N^*$ (positive integers),
$\kk=\Q$ and $\lam =$ the inclusion map $\N^* \hookrightarrow \Q$.
%
%
Formula~\eqref{eqfacileS} then reads
\[
S^{n_1\cdots n_r} = \frac{1}{
(n_1+n_2+\cdots+ n_r)(n_2+\cdots+ n_r)\cdots n_r}.
\]
In that case, the Hopf algebra $\kUN$ evoked in
footnote~\ref{ftn:shuff} is the combinatorial Hopf algebra
$\operatorname{QSym}$ of ``quasi-symmetric functions'' and this
mould~$S^\bul$ is related to the so-called Dynkin Lie idempotent,
of which we thus get interesting generalisations by considering arbitrary
maps $\lam \col \N^* \to \Q$ and the corresponding symmetral moulds~$S^\bul$.

The canonical case defined in Section~\ref{paragaCanCase} is the opposite:
$\Alt_0(\kk)$ is huge. Choosing a resonant alternal mould~$A^\bul$ amounts to
choosing an arbitrary constant in~$\kk$ for $A^0$ (only possibly
nonzero value in length~$1$), an arbitrary odd
function $\kk\to\kk$ for $\lam_1 \mapsto A^{\lam_1(-\lam_1)}$ in length~$2$, etc.


\paraga    \label{paragaJaugeK}
The exponential map induces a bijection from $\Alt\RES(\cN)$ to the
set $\Sym\RES(\cN)$ consisting of all $\lam$-resonant symmetral
moulds, which is a subgroup of $\Sym(\cN)$.

According to part~(ii) of Theorem~\ref{thmB},
given a solution $(F^\bul, S^\bul) \in \Alt(\cN)\times\Sym(\cN)$ of~\eqref{eqmouldFS},
%
%
we thus have a bijection
\be    \label{eqgaugetrsfKS}
K^\bul \mapsto (\ti F^\bul, \ti S^\bul) = 
\big( \inv K^\bul \times F^\bul \times K^\bul,
\, S^\bul \times K^\bul \big) 
\ee
between $\Sym\RES(\cN)$ and the set of all solutions
$(\ti F^\bul, \ti S^\bul) \in \Alt(\cN)\times\Sym(\cN)$
of~\eqref{eqmouldFS}.
The map
\[
(F^\bul, S^\bul) \mapsto
\big( \inv K^\bul \times F^\bul \times K^\bul,
\, S^\bul \times K^\bul \big) 
\]
is called the ``gauge transformation'' associated with $K^\bul \in \Sym\RES(\cN)$.

The group $\Sym\RES(\cN)$ is called the ``gauge group'' of
equation~\eqref{eqmouldFS}; it acts to the right freely and
transitively by gauge transformations on the space of solutions
$\big\{ (F^\bul,S^\bul) \big\} \subset \Alt(\cN)\times\Sym(\cN)$.
Its effect on gauge generators is given by the formula
\be   \label{eqnewjaugeSK}
\JJ(\log\ti S^\bul) = \inv K^\bul \times \JJ(\log S^\bul) \times K^\bul
+ \inv K^\bul \times \na_1 K^\bul.
\ee


\paraga
The identities
\[
\ex^{-J^\bul} \times A^\bul \times \ex^{J^\bul}  = 
(\ex^{-\ad_{J^\bul}}) A^\bul = \sum_{k\ge0} \tfrac{(-1)^k}{k!} (\ad_{J^\bul})^k A^\bul,
\qquad 
\ex^{-J^\bul} \times \na_\lam(\ex^{J^\bul}) = 
\sum_{k\ge0} \tfrac{(-1)^k}{(k+1)!} (\ad_{J^\bul})^k \, \na_\lam J^\bul
\]
to be seen in Section~\ref{subsecMouldExpCFLA} (Propositions~\ref{propLieMCun}(ii) and~\ref{propLieMCdeux}(ii))
show that, for any alternal mould~$M^\bul$, the $\lam$-resonant mould
$\JJ(M^\bul)$ is alternal,
as claimed in the paragraph following Definition~\ref{deflamrespart},
and that the \rhs\ of~\eqref{eqnewgauge} or~\eqref{eqnewjaugeSK} is
indeed alternal and $\lam$-resonant
(by replacing~$\na_\lam$ with~$\na_1$ and observing that
$\Alt\RES(\cN)$ is invariant by~$\ad_{J^\bul}$ for $J^\bul\in\Alt\RES(\cN)$).

One can easily find the gauge transformation which maps the zero gauge
solution on any given solution:
if a given solution $(F^\bul, S^\bul) \in \Alt(\cN)\times\Sym(\cN)$
has gauge generator $A^\bul=\JJ(\log S^\bul)$, then one finds the
desired gauge transformation in terms of~$A^\bul$ by solving the
equation
\[
\na_1 K^\bul = K^\bul \times A^\bul
\]
inductively on word length with initial condition $K^\est=1$
(the unique solution $K^\bul \in \kk^\UN$ is clearly $\lam$-resonant
and it turns out that it is also symmetral).

\newpage

\centerline{\Large\sc Lie mould calculus}
\addcontentsline{toc}{part}{\vspace{-2.1em} \\ \sc Lie mould calculus\vspace{.4em}}


\section{Lie mould calculus and proof of  Theorem~\ref{thmA}}   \label{secLMCpfThmA}

\subsection{General setting}   \label{secGenSetting}
$ $\\[-2.75ex]

\noindent
%
Let us give ourselves a field~$\kk$ and a nonempty set~$\cN$,
so that we can consider the associative $\kk$-algebra~$\kk^\UN$ of
Section~\ref{secDefMoulds}.
We suppose that we are also given a Lie algebra~$\cL$ over~$\kk$ and a
family $(B_n)_{n\in\cN}$ of~$\cL$.

Let us consider an associative algebra~$\A$ over~$\kk$ such that~$\cL$
is a Lie subalgebra of~$\LIE(\A)$ (we denote by~$\LIE(\A)$ the Lie
algebra over~$\kk$ with the same underlying vector
space as~$\A$ and bracketing $[x,y] \defeq xy-yx$).
For instance, by the Poincar\'e-Birkhoff-Witt theorem, we may take
for~$\A$ the universal enveloping algebra of~$\cL$.

\begin{definition}
  The ``associative comould'' is the family
  $B_\bul = (B_\un)_{\un\in\UN}$ defined by
\[
B_\un \defeq B_{n_r} \cdots B_{n_1}  \in \A
\]
for any word $\un = n_1\cdots n_r$,
with the convention $B_{\est} \defeq 1_\A$.
The ``Lie comould'' is the family $\Bul = (\Bun)_{\un\in\UN}$ defined
by
$B_{[\est]} \defeq 0$ and
\[
\Bun \defeq \ad_{B_{n_r}} \circ \cdots \circ \ad_{B_{n_2}} B_{n_1} 
= [B_{n_r},[\ldots[B_{n_2},B_{n_1}]\ldots]] \in \cL
\]
for any nonempty word $\un = n_1\cdots n_r$, 
with the convention $B_{[n_1]}=B_{n_1}$ when $r=1$.
\end{definition}


Beware that in general, contrarily to the Lie comould, the associative
comould is not a family of~$\cL$, but only of~$\A$.
\'Ecalle's mould calculus (\cite{Eca81}, \cite{Eca93}, \cite{mouldSN})
deals with finite or infinite sums of the form $\sum M^\un B_\un$ in
the associative algebra~$\A$, with arbitrary moulds
$M^\bul \in \kk^\UN$.
In this article, we use the phrase ``Lie mould calculus'' when
restricting our attention to finite or infinite sums of the form
$\sum M^\un B_\un$ with \emph{alternal} moulds~$M^\bul$ because, as
will be shown in a moment, such expressions can be rewritten
$\sum \frac{1}{r(\un)} M^\un \Bun$ and thus belong to the Lie
algebra~$\cL$.


The shuffling coefficients of Definition~\ref{def:shufflcoeff} allow us to
express the Lie comould~$\Bul$ in terms of the associative
comould~$B_\bul$:

\begin{lemma}   \label{lemlimkLieAssComould}
For any nonempty word $\un \in \UN$,
\[
\Bun = \sum_{(\ua,\ub) \in \UN\times\UN} (-1)^{r(\ub)} r(\ua) 
\, \shabn \, B_{\,\wt\ub \, \ua},
\]
where, for an arbitrary word $\ub = b_1\cdots b_s$, we denote
by~$\wt\ub$ the reversed word: $\wt\ub = b_s \cdots b_1$.
\end{lemma}


\begin{proof}
Let us show by induction on~$r$ that
\be   \label{eqinducshabn}
\sum_{(\ua,\ub) \in \UN\times\UN} (-1)^{r(\ub)} 
\shabn \, B_{\,\wt\ub \, \ua}
= 0, \qquad
\sum_{(\ua,\ub) \in \UN\times\UN} (-1)^{r(\ub)} r(\ua) 
\, \shabn \, B_{\, \wt\ub \, \ua}
= \Bun
\ee
for any word~$\un$ of length $r\ge1$.
We denote the first sum by $\LHS\un$ and the second by $\LHSp\un$,
and observe that, as a consequence of~\eqref{eqdefshabn},
\be   \label{eqdefLHSLHSp}
\LHS\un = \sum_{\ell=0}^r \, \sum_{\tau\in\fS_r(\ell)} (-1)^{r-\ell}
B_{\,\wt{\phantom{n^\tau}}\hspace{-1em}\un^\tau\usup\ell \, \un^\tau\uleq\ell},
\qquad
\LHSp\un = \sum_{\ell=0}^r \, \sum_{\tau\in\fS_r(\ell)} (-1)^{r-\ell} \ell\,
B_{\,\wt{\phantom{n^\tau}}\hspace{-1em}\un^\tau\usup\ell \, \un^\tau\uleq\ell}
\ee

For $r=1$, we find $\LHS{n_1} = B_{n_1}-B_{n_1}=0$ 
and $\LHSp{n_1} = 1\cdot B_{n_1}-0\cdot B_{n_1}=B_{[n_1]}$.

Let us assume that $r\ge2$ and~\eqref{eqinducshabn} holds for any
word~$\un$ of length $r-1$.
Given an arbitrary word~$\um$ of length~$r$, we write it as 
$\um = \un \, c$, where $\un\in\cN^{r-1}$ and $c\in\cN$.
When using~\eqref{eqinducshabn} to compute $\LHS\um$ or
$\LHSp\um$, we see that the last letter of~$\um$ must either go
at the end of~$\ub$ or at the end of~$\ua$, or, more precisely, 
using~\eqref{eqdefLHSLHSp}, we see that $\fS_r(\ell)$ can be written
as a disjoint union
\[
\fS_r(\ell) = \fB \sqcup\fA, \qquad
\fB \defeq \{\, \tau\in \fS_r(\ell) \mid \tau(r) = r\,\}, \quad
\fA \defeq \{\, \tau\in \fS_r(\ell) \mid \tau(r) < r\,\}
\]
(note that $\tau\in\fA \;\Rightarrow\; 1\le \ell < r$ and $\tau(\ell)=r$),
and there are bijections $\tau\in\fB \mapsto \tau'\in\fS_{r-1}(\ell)$
and $\tau\in\fA \mapsto \tau^*\in\fS_{r-1}(\ell-1)$ (note that
$\fA$ is empty when $\ell=0$) so that
\[
\um^\tau\uleq\ell = \un^{\tau'}\uleq\ell
\ens\text{and}\ens
\um^\tau\usup\ell = \un^{\tau'}\usup\ell\, c
\ens\text{for $\tau\in\fB$}, 
\qquad
\um^\tau\uleq\ell = \un^{\tau^*}\uleq{\ell-1}\, c
\ens\text{and}\ens
\um^\tau\usup\ell = \un^{\tau^*}\usup{\ell-1}
\ens\text{for $\tau\in\fA$}
\]
(namely $\tau'(i)=\tau(i)$ for $1\le i \le r-1$, and
$\tau^*(i) = \tau(i)$ for $i\le \ell-1$
while $\tau^*(i) = \tau(i+1)$ for $\ell\le i \le r-1$).\footnote{%
Another way of seeing this is to consider the ``unshuffling
coproduct'' on the vector space~$\kUN$ of footnote~\ref{ftn:dualkUN}:
this is the linear map $\De \col \kUN \to \kUN\otimes\kUN$ determined
by
$\De(\un) = \sum \shabn \, \ua \otimes \ub$,
and the above property amounts to the inductive definition
$\De(\est)=0$ and $\De(\un \,c) = \De(\un) (\est\otimes c + c \otimes \est)$,
where we make use of the non-commutative associative ``concatenation product''
on~$\kUN$ or $\kUN\otimes\kUN$
(in fact, this gives rise to another Hopf algebra structure on~$\kUN$).
}
Therefore
\[
\LHS\um = \sum_{(\ua,\ub) \in \UN\times\UN} (-1)^{r(\ub\, c)} 
\shabn \, B_{c\,\wt\ub \, \ua}
+
\sum_{(\ua,\ub) \in \UN\times\UN} (-1)^{r(\ub)} 
\shabn \, B_{\,\wt\ub \, \ua\, c}
\]
and, since $B_{c\,\wt\ub \, \ua}=B_{\,\wt\ub \, \ua}B_c$ and 
$B_{\,\wt\ub \, \ua\, c} = B_c B_{\,\wt\ub \, \ua}$, we get
$\LHS\um = - \LHS\un B_c + B_c \LHS\un = 0$
by the induction hypothesis; on the other hand,
\begin{align*}
\LHSp\um &= \sum_{(\ua,\ub) \in \UN\times\UN} (-1)^{r(\ub\, c)}  r(\ua) \,
\shabn \, B_{c\,\wt\ub \, \ua}
+
\sum_{(\ua,\ub) \in \UN\times\UN} (-1)^{r(\ub)}  r(\ua\,c)\, 
\shabn \, B_{\,\wt\ub \, \ua\, c} \\[1ex]
&= - \LHSp\un B_c + B_c \big( \LHSp\un + \LHS\un \big)
= [ B_c,\Bun ] = B_{[\un\,c]} = \Bum.
\end{align*}
\end{proof}

\subsection{Finite mould expansions}   \label{secLieMouldExp}
$ $\\[-2.75ex]

\noindent
%
Let us denote by $\kkfUN$ the set of finite-support moulds, which is
clearly an associative subalgebra of $\kk^\UN$.
The finiteness condition allows us to define a map with values in~$\A$
by means of the associative comould~$B_\bul$:
\be    \label{eqdeffcontr}
M^\bul \in \kkfUN \mapsto 
M^\bul B_\bul \defeq
\sum_{\un\in\UN} M^\un \, B_\un \in \A.
\ee
Since $B_{\ua\,\ub} = B_\ub B_\ua$ for any two words $\ua,\ub$, it is
obvious that the map~\eqref{eqdeffcontr} is an associative algebra
anti-morphism, \ie
\be   \label{eqantimorphassalg}
(M^\bul \times N^\bul) B_\bul =
(N^\bul B_\bul) (M^\bul B_\bul)
\quad\text{for any $M^\bul,N^\bul \in \kkfUN$.}
\ee


We can also define a map with values in~$\cL$ by means of the Lie
comould~$\Bul$:
\be   \label{eqdefContrBul}
M^\bul \in \kkfUN \mapsto M^\bul \Bul \defeq 
\sum_{\un\neq\est} \tfrac{1}{r(\un)} M^\un \, \Bun \in \cL.
\ee


\begin{lemma}    \label{lemMouldExpAlft}
Let $M^\bul \in \Alt(\cN)$ and let~$\Om$ be an orbit of the
action~\eqref{eqrightactionfSr} of~$\fS_r$ for some $r\in\N^*$.
Then
\be   \label{eqMouldExpOneOrb}
\sum_{\un\in\Om} M^\un \, \Bun =
r \sum_{\un\in\Om} M^\un \, B_{\un}.
\ee
If $M^\bul \in \Alt(\cN)\cap\kkfUN$, then
\be   \label{eqMouldExpAltf}
M^\bul B_\bul = M^\bul \Bul.
\ee
\end{lemma}


\begin{proof}
  Lemma~\ref{lemlimkLieAssComould} allows us to rewrite the \lhs\
  of~\eqref{eqMouldExpOneOrb} as
\[
\LhS = \sum_{(\ua,\ub) \in \UN\times\UN} (-1)^{r(\ub)} r(\ua) 
\bigg(\sum_{\un\in\Om} \shabn M^\un \bigg)
B_{\,\wt\ub \, \ua}.
\]
In view of~\eqref{eqdefshabn}, the sum between parentheses is~$0$ if
$\ua\,\ub\notin\Om$, whereas, if $\ua\,\ub\in\Om$, it is
\[ \sum_{\un\in\UN} \shabn M^\un. \]
According to Definition~\ref{defalt}, the latter sum is~$0$ when
both~$\ua$ and~$\ub$ are nonempty, and it is $M^\ua$ when $\ub=\est$,
hence we end up with
$\LhS = \sum_{\ua \in \Om} r(\ua) M^\ua \, B_\ua$,
which coincides with the \rhs\ of~\eqref{eqMouldExpOneOrb}.

To prove~\eqref{eqMouldExpAltf}, by linearity we can
assume that there is $r\ge1$ such that the support of~$M^\bul$ is
contained in~$\cN^r$.
Then we can partition~$\cN^r$ into orbits:
\[
M^\bul B_\bul = \sum_{\Om\in \cN^r / \fS_r} \; \sum_{\un\in\Om} \,
M^\un \, B_\un
= \tfrac{1}{r} \sum_{\Om\in \cN^r / \fS_r} \; \sum_{\un\in\Om} \,
M^\un \, \Bun
= \tfrac{1}{r} \sum_{\un\in\cN^r} M^\un \, \Bun
= \sum_{\un\neq\est} \tfrac{1}{r(\un)} M^\un \, \Bun.
\]
\end{proof}


\begin{remark}
An identity more precise than~\eqref{eqMouldExpOneOrb} is mentioned in
\'Ecalle's works:
given a letter~$c$ and an orbit~$\Om$ of the
action~\eqref{eqrightactionfSr} of~$\fS_r$ for some $r\in\N^*$, 
let $r_c(\Om)$ denote the number of occurrences of the letter~$c$ in
any word of~$\Om$ and let 
$\Om_c \defeq \{\, \un \in \Om \mid n_1 = c \,\}$;
then, for any alternal mould~$M^\bul$, 
\[
\sum_{\un\in\Om_c} M^\un \, \Bun = r_c(\Om) \sum_{\un\in\Om} M^\un \, B_\un.
\]
This is related to the identity
\[
B_{[c\,\un]} = \sum_{(\ua,\ub) \in \UN\times\UN} (-1)^{r(\ub)} 
\, \shabn \, B_{\,\wt\ub \, c \, \ua}
\quad\text{for any $c\in\cN$ and $\un\in\UN$}
\]
and to the following consequence of alternality:
\[
M^{\ua\, c\, \ub} = (-1)^{r(\ua)} \sum_{\un\in\UN} \shtiabn M^{c\,\un}
\quad \text{for any $c\in\cN$ and $\ua,\ub\in\UN$}
\]
(stated as formula~(5.26) in~\cite{EV95}).
\end{remark}


Recall that, as mentioned in Section~\ref{secalternalmoulds}, the
set $\Alt(\cN)$ of alternal moulds is a Lie subalgebra of $\LIE(\kk^\UN)$.
Let us denote the set of finite-support alternal moulds by
\[
\Altf(\cN) \defeq \Alt(\cN) \cap \kkfUN.
\]
It is obvious that $\Altf(\cN)$ is also a Lie subalgebra.
In view of~\eqref{eqMouldExpAltf}, there is no need to distinguish
between the maps~\eqref{eqdeffcontr} and~\eqref{eqdefContrBul} when
restricting to~$\Altf(\cN)$.


\begin{proposition}   \label{propLieAntimAltfcL}
  The map $M^\bul \mapsto M^\bul \Bul$ induces a Lie algebra anti-morphism
  $\Altf(\cN) \to \cL$, \ie
\[
[M^\bul,N^\bul] \Bul =
[N^\bul \Bul, M^\bul \Bul]
\quad\text{for any $M^\bul,N^\bul \in \Altf(\cN)$.}
\]
\end{proposition}


\begin{proof}
Using~\eqref{eqantimorphassalg} and~\eqref{eqMouldExpAltf}, we compute
$[M^\bul,N^\bul] \Bul = [M^\bul,N^\bul] B_\bul = 
(M^\bul \times N^\bul) B_\bul - (N^\bul \times M^\bul) B_\bul
= (N^\bul B_\bul) (M^\bul B_\bul) - (M^\bul B_\bul) (N^\bul B_\bul)
= [N^\bul B_\bul, M^\bul B_\bul] = [N^\bul \Bul, M^\bul \Bul]$.
\end{proof}


\begin{proposition}   \label{prop:lamBbulBun}
Suppose that there are a function $\lam\col\cN\to\kk$ and an
$X_0\in\cL$ such that $[X_0,B_n] = \lam(n) B_n$ for each letter~$n$.
Then
\be   \label{eqcrochXzMbBb}
[X_0,M^\bul B_\bul] = (\na_\lam M^\bul) B_\bul
\ens\;\text{and}\ens\;
[X_0,M^\bul \Bul] = (\na_\lam M^\bul) \Bul
\quad\text{for any $M^\bul\in\kkfUN$.}
\ee
\end{proposition}

\begin{proof}
One easily checks that
\[
[X_0,B_\un] = \lam(\un) B_\un
\ens\;\text{and}\ens\;
[X_0,\Bun] = \lam(\un) \Bun
\quad \text{for any $\un\in\UN$}
\]
by induction on $r(\un)$ (because $[X_0,\cdot\,]$ is a derivation of
the associative algebra~$\A$, as well as derivation of the Lie algebra~$\cL$),
whence~\eqref{eqcrochXzMbBb} follows.
\end{proof}

\subsection{Mould expansions in complete filtered Lie algebras}   \label{subsecMouldExpCFLA}
$ $\\[-2.75ex]

\noindent
%
We now assume that $\cL$ is a complete filtered Lie
algebra and that $(B_n)_{n\in\cN}$ is a formally summable family
such that each~$B_n$ has order~$\ge1$.
We do not need any auxiliary associative algebra~$\A$ such that $\cL
\subset \LIE(\A)$ in this section, except at the end of Remark~\ref{remdefsym}.


\begin{lemma}
For each $M^\bul \in \kk^\UN$ the family $(\tfrac{1}{r(\un)} M^\un \,
\Bun)_{\un\neq\est}$ is formally summable,
hence there is a well-defined extension of the
map~\eqref{eqdefContrBul} to the set of all moulds (for which we use
the same notation):
\be   \label{eqdefgencontr}
M^\bul \in \kk^\UN \mapsto 
M^\bul \Bul \defeq
\sum_{\un\neq\est} \tfrac{1}{r(\un)} M^\un \, \Bun \in \cL.
\ee
This is a $\kk$-linear map, compatible with the filtrations of~$\kk^\UN$ and~$\cL$ in
the sense that, for each $m\in\N$ and $M^\bul \in \kk^\UN$,
\be   \label{ineqcompatord}
\ord(M^\bul)\ge m \imp \ord(M^\bul\Bul)\ge m
\ee
(with the notation~\eqref{eqdefordmould} for the order function associated with the
filtration of~$\kk^\UN$).
\end{lemma}


\begin{proof}
By assumption, $\cN_m \defeq \{\, n \in\cN \mid \ord(B_n) < m \,\}$ is
finite for each $m\in\N$ and, in view of Definition~\ref{DefCFLA},
$\ord(\Bun) \ge r(\un)$ for each $\un\in \UN$. This implies that
\[
\{\, \un \in \UN \mid \ord(\Bun) < m \,\} \subset
\{\, \un \in \UN \mid r\defeq r(\un)<m \;\text{and}\; n_1,\ldots,n_r\in\cN_m \,\},
\]
which is finite, hence the formal summability follows.
The property~\eqref{ineqcompatord} is obvious.
\end{proof}


Note that, if $M^\est=0$ (as is the case when~$M^\bul$ is alternal),
then $\ex^{M^\bul}$ is a well-defined mould and $Y\defeq M^\bul\Bul$
has order $\ge1$, hence $\ex^{\ad_Y}$ is a well-defined Lie algebra automorphism.


\begin{proposition}   \label{propLieMCun}
(i)
The map~\eqref{eqdefgencontr} induces a Lie
algebra anti-morphism $\Alt(\cN) \to \cL$, \ie
\be   \label{eqLieAntimAltcL}
[M^\bul,N^\bul] \Bul =
[N^\bul \Bul, M^\bul \Bul]
\quad\text{for any $M^\bul,N^\bul \in \Alt(\cN)$.}
\ee
(ii)
If $M^\bul,N^\bul \in \Alt(\cN)$, then the mould
$\ex^{-M^\bul} \times N^\bul \times \ex^{M^\bul}$ can be written
\[
\ex^{-M^\bul} \times N^\bul \times \ex^{M^\bul} = 
\big( \ex^{-\ad_{M^\bul}}\big) N^\bul
= \sum_{k\ge0} \tfrac{(-1)^k}{k!} (\ad_{M^\bul})^k N^\bul
\]
and is alternal, and $Y \defeq M^\bul \Bul$ satisfies
\[
\ex^{\ad_Y} (N^\bul \Bul) = 
\big( \ex^{-M^\bul} \times N^\bul \times \ex^{M^\bul} \big) \Bul.
\]
\end{proposition}


\begin{proof}
(i) As mentioned in Section~\eqref{secDefMoulds}, the set of all
moulds $\kk^\UN$ is a complete metric space for the distance
$d(M^\bul,N^\bul) \defeq 2^{-\ord(N^\bul-M^\bul)}$.
The map $M^\bul \mapsto M^\bul\Bul$
is continuous (and even $1$-Lipschitz) by~\eqref{ineqcompatord}, and
the set of finite-support alternal moulds $\Altf(\cN)$ is dense in 
$\Alt(\cN)$, so~\eqref{eqLieAntimAltcL} follows from Proposition~\ref{propLieAntimAltfcL}.

(ii) Because of~(i), the adjoint representations of $\Alt(\cN)$
and~$\cL$ are related by
\be   \label{eqreladjreprAltcL}
M^\bul\in\Alt(\cN),\; Y = M^\bul\Bul \imp
\ad_Y(N^\bul\Bul) = - (\ad_{M^\bul}N^\bul) \Bul
\ens\text{for any $N^\bul\in\Alt(\cN)$,}
\ee
therefore 
$\ex^{\ad_Y}(N^\bul\Bul) = \big( \ex^{-\ad_{M^\bul}}(N^\bul) \big) \Bul$,
where $\ex^{-\ad_{M^\bul}}(N^\bul) \in \Alt(\cN)$ is well-defined
because $M^\est=0$, hence $\ad_{M^\bul}$ increases order in~$\cL$ by
at least one unit and $\ex^{-\ad_{M^\bul}}$ is a well-defined
$\kk$-linear operator of~$\Alt(\cN)$.

In fact, $\Alt(\cN) \hookrightarrow \LIE(\kk^\UN)$ and
$\ex^{-\ad_{M^\bul}}$ is also a well-defined $\kk$-linear operator
of~$\kk^\UN$; as such, it can be written
\[
\ex^{-\ad_{M^\bul}} = \ex^{-L_{M^\bul}+R_{M^\bul}} 
= \ex^{-L_{M^\bul}} \circ \ex^{R_{M^\bul}},
\]
where $L_{M^\bul}, R_{M^\bul} \in \End_\kk\big( \kk^\UN \big)$ are the
operators of left-multiplication and right-multiplication by~$M^\bul$,
which commute.
Obviously, $\ex^{-L_{M^\bul}}$ and $\ex^{R_{M^\bul}}$ are the
operators of left-multiplication and right-multiplication
by~$\ex^{-M^\bul}$ and~$\ex^{M^\bul}$, hence
$\ex^{-\ad_{M^\bul}}(N^\bul) = \ex^{-M^\bul} \times N^\bul \times \ex^{M^\bul}$
(the latter identity is sometimes called Hadamard lemma; we gave these
details because later we will need again the operators~$L_{M^\bul}$ and~$R_{M^\bul}$).
\end{proof}


\begin{proposition}   \label{propLieMCdeux}
Suppose that there are a function $\lam\col\cN\to\kk$ and an
$X_0\in\cL$ such that $[X_0,B_n] = \lam(n) B_n$ for each letter~$n$.
If $M^\bul \in \Alt(\cN)$, then
\begin{enumerate}[(i)]
\item
the mould $\na_\lam M^\bul$ is alternal and 
\be   \label{eqcrochXzMbBul}
[ X_0, M^\bul\Bul ] = (\na_\lam M^\bul) \Bul,
\ee
\item
the mould $\ex^{-M^\bul} \times \na_\lam (\ex^{M^\bul})$ can be
written
\[
\ex^{-M^\bul} \times \na_\lam (\ex^{M^\bul}) = 
\sum_{k\ge0} \tfrac{(-1)^{k}}{(k+1)!} (\ad_{M^\bul})^k \na_\lam M^\bul
\]
and is alternal, and $Y\defeq M^\bul\Bul$ satisfies
\[
\ex^{\ad_Y} X_0 = X_0 - 
\big( \ex^{-M^\bul} \times \na_\lam (\ex^{M^\bul}) \big) \Bul.
\]
\end{enumerate}
\end{proposition}


\begin{proof}
(i)
The identity~\eqref{eqcrochXzMbBul} holds for any $M^\bul \in
\kk^\UN$, as a consequence of~\eqref{eqcrochXzMbBb}, by continuity of
$M^\bul \mapsto M^\bul \Bul$ and density of~$\kkfUN$ in~$\kk^\UN$.
It is obvious that~$\na_\lam$ preserves alternality.

(ii)
We write $\ex^{\ad_Y}X_0 - X_0 
= \sum_{k\ge0} \frac{1}{(k+1)!} (\ad_Y)^{k+1} X_0 
= - \sum_{k\ge0} \frac{1}{(k+1)!} (\ad_Y)^k [X_0,Y]$
with $[X_0,Y] = (\na_\lam M^\bul)\Bul$ by~\eqref{eqcrochXzMbBul}, whence
$ (\ad_Y)^k [X_0,Y] = (-1)^{k} \big( (\ad_{M^\bul})^k \na_\lam M^\bul) \big) \Bul$
by~\eqref{eqreladjreprAltcL}. Therefore
\be   \label{eqexadYXzmXz}
\ex^{\ad_Y}X_0 - X_0 = -(P \, \na_\lam M^\bul) \Bul
\quad\text{with}\ens
P \defeq \sum_{k\ge0} \tfrac{(-1)^{k}}{(k+1)!} (\ad_{M^\bul})^k
\in \End_\kk\big(\kk^\UN\big).
\ee
Note that~$P$ is a well-defined $\kk$-linear operator of~$\kk^\UN$ which
preserves $\Alt(\cN)$,
because $\ad_{M^\bul}$ increases order in~$\kk^\UN$ by at least one
unit and preserves~$\Alt(\cN)$.

On the other hand, as~$\na_\lam$ is a derivation of the associative
algebra~$\kk^\UN$, the Leibniz formula applied to 
$\ex^{M^\bul} = 1^\bul + \sum_{k\ge0} \frac{1}{(k+1)!} (M^\bul)^{\times(k+1)}$
yields
\[
\na_\lam ( \ex^{M^\bul} ) = 
\sum_{k\ge0} \tfrac{1}{(k+1)!} \sum_{p+q=k} 
(M^\bul)^{\times p} \times \na_\lam M^\bul \times (M^\bul)^{\times q}
= \sum_{k\ge0} \tfrac{1}{(k+1)!} \sum_{p+q=k} 
L_{M^\bul}^p R_{M^\bul}^q  (\na_\lam M^\bul),
\]
with the same left- and right-multiplication operators~$L_{M^\bul}$
and~$R_{M^\bul}$ as in the end of the proof of Proposition~\ref{propLieMCun}.
Left-multiplication by $\ex^{-M^\bul}$ coincides with the
operator $\ex^{-L_{M^\bul}}$, therefore
\be    \label{eqdefoperQ}
\ex^{-M^\bul} \times \na_\lam ( \ex^{M^\bul} ) = 
Q\, \na_\lam M^\bul
\quad\text{with}\ens
Q \defeq \ex^{-L_{M^\bul}} \sum_{k\ge0} \tfrac{1}{(k+1)!} \sum_{p+q=k} 
L_{M^\bul}^p R_{M^\bul}^q
\in \End_\kk\big(\kk^\UN\big).
\ee

Since $\ad_{M^\bul} = L_{M^\bul} - R_{M^\bul}$, we see that
$P=Q$ in $\End_\kk\big(\kk^\UN\big)$, as a consequence of the following
identity between (commutative) series of two indeterminates:
\[
\sum_{k\ge0} \tfrac{(-1)^{k}}{(k+1)!} (L-R)^k =
\ex^{-L} \sum_{k\ge0} \tfrac{1}{(k+1)!} \sum_{p+q=k} L^p R^q
\in \Q[[L,R]]
\]
(which can be checked, since $\Q[[L,R]]$ has no divisor of zero,
by multiplying both sides by $L-R$: the \lhs\ yields
$-\ex^{-L+R}+1$ and the \rhs\ yields 
$\ex^{-L} \sum_{k\ge0} \tfrac{1}{(k+1)!} (L^{k+1}-R^{k+1}) 
= \ex^{-L} (\ex^L-\ex^R)$).

Since $P=Q$, \eqref{eqdefoperQ} shows that
$\ex^{-M^\bul} \times \na_\lam ( \ex^{M^\bul} ) = 
P\, \na_\lam M^\bul \in \Alt(\cN)$ (because $\na_\lam M^\bul$ is
alternal and~$P$ preserves~$\Alt(\cN)$),
and~\eqref{eqexadYXzmXz} yields
$\ex^{\ad_Y}X_0 - X_0 = -\big(\ex^{-M^\bul} \times \na_\lam ( \ex^{M^\bul} )\big) \Bul$.
\end{proof}


\begin{remark}   \label{remdefsym}
  The set $\Sym(\cN) \subset \kk^\UN$ of symmetral moulds has been
  defined in~\eqref{defSymexpAlt} as the set of all exponentials of
  alternal moulds.
Here is a characterization more in the spirit of
Definition~\ref{defalt} (the proof of which can be found \eg in \cite[Prop.~5.1]{mouldSN}):
\emph{
A mould~$M^\bul$ is symmetral if and only if}
\be   \label{eqdefstdsym}
M^\est=1 \quad\text{\emph{and}}\quad
\sum_{\un\in\UN} \shabn M^\un = M^\ua M^\ub \quad\text{\emph{for any
two nonempty words $\ua,\ub$.}}
\ee
When identifying~$\kk^\UN$ with the dual of~$\kUN$ as in
footnotes~\ref{ftn:dualkUN} and~\ref{ftn:shuff}, we thus identify
the symmetral moulds with the characters of the associative algebra
$(\kUN,\shuff)$, \ie when viewed as linear forms of~$\kUN$ they are
characterised by
$M^\bul(x\shuff y) = M^\bul(x) M^\bul(y)$.
In that point of view, $\Sym(\cN)$ is a group because $\kUN$ is a bialgebra.

In the case when $\cL \hookrightarrow \LIE(\A)$, where~$\A$ is a
complete filtered associative algebra such that
$\cL\ugeq m = \cL \cap \A\ugeq m$ for each~$m$, the
map~\eqref{eqdeffcontr} extends to an associative algebra
anti-morphism
$M^\bul \in \kk^\UN \mapsto M^\bul B_\bul \in \A$,
compatible with the filtrations of~$\kk^\UN$ and~$\A$, 
whose restriction to $\Alt(\cN)$ coincide with that of $M^\bul \mapsto
M^\bul \Bul$.
Then
\[
M^\est = 0 \imp
\ex^{M^\bul B_\bul} = \big(\ex^{M^\bul}\big) B_\bul.
\]
In particular, if $M^\bul$ is alternal, then $\ex^{M^\bul \Bul} = \big(\ex^{M^\bul}\big) B_\bul$
with $\ex^{M^\bul}$ symmetral.
\end{remark}


\subsection{Theorem~\ref{thmB} implies Theorem~\ref{thmA}}   \label{secPfThmA}
$ $\\[-2.75ex]

\noindent
%
In this section, we take Theorem~\ref{thmB} for granted and show how
Theorem~\ref{thmA} follows from Lie mould calculus.
We thus assume that we are given $\cN$ a nonempty set, 
$\kk$ a field of characteristic zero,
$\lam\col\cN\to\kk$ a map,
$\cL$ a complete filtered Lie algebra over~$\kk$, 
an element $X_0 \in \cL$,
and a formally summable family $(B_n)_{n\in\cN}$ such that $\ord(B_n)
\ge 1$ and $[X_0,B_n] = \lam(n) B_n$ for each $n\in\cN$.

Let us consider any of the many solutions
$(F^\bul,G^\bul) \in \Alt(\kk) \times \Alt(\kk)$ of
equation~\eqref{eqmouldFG} that Theorem~\ref{thmB} provides in the
canonical case of Section~\ref{paragaCanCase},
\ie with~$\na_{\id}$ replacing~$\na_\lam$.
We thus have alternal moulds $F^\bul,G^\bul$,
explicitly defined by \eqref{eqinducinit}--\eqref{eqGlogS} with some
$A^\bul \in \Alt_0(\kk)$, which satisfy equation~\eqref{eqmouldFG}.

Using the map $\lam^* \col \kkukk \to \kk^\UN$ of
Section~\ref{paragaphcNcM}, we define
$F_\lam^\bul \defeq \lam^*(F^\bul)$ and
$G_\lam^\bul \defeq \lam^*(G^\bul)$,
which belong to $\Alt(\cN)$ and satisfy equation~\eqref{eqmouldFG} but
now with the operator~$\na_\lam$ associated with the eigenvalue
map~$\lam$.

Let $Z \defeq F_\lam^\bul \Bul$, in accordance with~\eqref{eqmouldexpZ}. We
have $Z\in \cL\ugeq1$ and
the first part of~\eqref{eqmouldFG} says that $\na_\lam F_\lam^\bul = 0$,
hence $[X_0,Z]=0$ by Proposition~\ref{propLieMCdeux}(i).

Let $Y \defeq G_\lam^\bul \Bul$, in accordance with~\eqref{eqmouldexpY}. We
have $Y\in \cL\ugeq1$ and 
the second part of~\eqref{eqmouldFG} can be rewritten
\[
- \ex^{-G_\lam^\bul} \times \big(\na_\lam(\ex^{G_\lam^\bul})\big)
+ \ex^{-G_\lam^\bul} \times I^\bul \times \ex^{G_\lam^\bul} 
= F_\lam^\bul.
\]
Let us apply the map $M^\bul \mapsto M^\bul\Bul$ to both sides:
because of Proposition~\ref{propLieMCun}(ii) and Proposition~\ref{propLieMCdeux}(ii),
the image of the \lhs\ is
$\ex^{\ad_Y} X_0 - X_0 + \ex^{\ad_Y}(I^\bul\Bul)$, 
while the image of the \rhs\ is~$Z$, we thus get
\[
\ex^{\ad_Y} \big( X_0  + I^\bul\Bul \big) = X_0 + Z,
\]
which is the desired result, since $I^\bul\Bul = \sum_{n\in\cN} B_n$
by~\eqref{eqdefmouldI}.

\subsection{Proof of the formulas~\eqref{eqadYGB}--\eqref{eqexpadYSB}
of Remark~\ref{remexpadSB}}  \label{secPfRem}
$ $\\[-2.75ex]

\noindent
%
We keep the same assumptions and notations as in Section~\ref{secPfThmA}.

Let us denote by $\cE \defeq \End_\kk(\cL)$ the associative algebra consisting of
all $\kk$-linear operators of the vector space underlying~$\cL$
(multiplication being defined as operator composition), and by
$\D$ the subset of all derivations of the Lie
algebra~$\cL$, which is in fact a Lie subalgebra of
$\LIE(\cE)$
(Lie bracket being defined as operator commutator).
For each $m\in\N$, we set
\be   \label{eqfiltrderiv}
\cE\ugeq m \defeq \{\,
T \in \cE \mid T(\cL\ugeq p) \subset \cL\ugeq{p+m}
\; \text{for each $p\in\N$}\,\}, \quad
\D\ugeq m \defeq \D \cap \cE\ugeq m.
\ee
It is easy to check that $\cE\ugeq0\supset\cE\ugeq1\supset\cE\ugeq2 \supset \ldots$ is a complete filtered associative algebra
and $\D\ugeq 0\supset\D\ugeq1\supset\D\ugeq2 \supset \ldots$ is a complete filtered Lie algebra. Moreover,
$\ad \col \cL \to \D\ugeq 0$ is a Lie algebra morphism
compatible with the filtrations, in the sense that it maps~$\cL\ugeq m$
to~$\D\ugeq m$.
Thus, $(\ad_{B_n})_{n\in\N}$ is a formally summable family contained
in $\D\ugeq 1$ and we are in the situation described at the end of
Remark~\ref{remdefsym}: with the notation $T_n\defeq \ad_{B_n}$, we
may consider the corresponding associative comould and Lie comould,
defined by
\[
T_\un \defeq \ad_{B_{n_r}} \cdots \ad_{B_{n_1}} \in \cE\ugeq r, \qquad
\Tun \defeq [\ad_{B_{n_r}},[\ldots[\ad_{B_{n_2}},\ad_{B_{n_1}}]\ldots]] =
\ad_{\Bun} \in \D\ugeq r
\]
for any $\un = n_1 \cdots n_r \in \UN$
(the identity $\Tun = \ad_{\Bun}$ is due to the Lie algebra morphism property). 
It follows that
$\ad_{M^\bul\Bul} = M^\bul \Tul$
for any $M^\bul\in\kk^\UN$
and, in the case of the alternal mould~$G_\lam^\bul$, 
\[
\ad_Y = \ad_{G_\lam^\bul\Bul} = G_\lam^\bul \Tul = G_\lam^\bul T_\bul
\]
because the restrictions to~$\Alt(\cN)$ of the maps $M^\bul\mapsto
M^\bul T_\bul$ and $M^\bul\mapsto M^\bul \Tul$ coincide. 
This is~\eqref{eqadYGB}.
Remark~\ref{remdefsym} also says that
\[
\ex^{G_\lam^\bul \Tul} = \big(\ex^{G_\lam^\bul}\big) T_\bul
\]
and, setting 
$S_\lam^\bul \defeq \lam^*(\ex^{G^\bul}) = \ex^{G_\lam^\bul}$
(recall that $\lam^* \col \kkukk \to \kk^\UN$ is a morphism of associative algebras), 
we get $\ex^{\ad_Y} = S_\lam^\bul T_\bul$,
which is~\eqref{eqexpadYSB}.

\subsection{Proof of the addendum to Theorem~\ref{thmA}}  \label{secPfAddend}
$ $\\[-2.75ex]

\noindent
%
We keep the same assumptions and notations as in
Section~\ref{secPfThmA}, except that now 
$F^\bul,G^\bul\in\Alt(\cN)$ are moulds satisfying~\eqref{eqmouldFG}
(\eg the ones denoted by $\lam^*(F^\bul)$ and $\lam^*(G^\bul)$ in Section~\ref{secPfThmA}).

Let $m\in\N^*$. The set
\[
\cN_m \defeq \{\, n\in\cN \mid
\ord(B_n) < m \, \}
\]
is finite, as a consequence of the formal summability of the family
$(B_n)_{n\in\cL}$.
We can thus define a ``truncation map'' 
$M^\bul \in \kk^\UN \mapsto M^\bul\uim \in \kkfUN$
by the formula
\[
M^\est\uim \defeq M^\est, \qquad
M^\un\uim \defeq 
\IND{r<m} \, \IND{n_1,\ldots,n_r\in\cN_m} \, M^\un
\quad\text{for any nonempty word $\un={n_1\cdots n_r} \in\UN$}
\]
and, in our current notations, the formulas \eqref{eq:defZm}--\eqref{eq:defYm} become
\begin{align*}
\tr{Z}{m} &\defeq \sum_{r= 1}^{m-1} \,\sum_{n_1,\ldots,n_r\in\cN_m}
\frac{1}{r} F^{n_1,\ldots,n_r} \Bun
= F^\bul\uim \Bul
\\[.6ex]
\tr{Y}{m} &\defeq\sum_{r=1}^{m-1} \, \sum_{n_1,\ldots,n_r\in\cN_m}
\frac{1}{r} G^{n_1\ldots,n_r} \Bun
= G^\bul\uim \Bul.
\end{align*}
Clearly $\na_\lam F^\bul=0$ entails $\na_\lam F^\bul\uim=0$, hence
$[X_0,Z_m]=0$ by Proposition~\ref{prop:lamBbulBun}.
It only remains to be proved that
\[
W_m \defeq
\ex^{\ad_{\tr{Y}{m}}} \Big(X_0+\sum_{n\in\cN}B_n\Big)
-X_0 - \tr{Z}{m}
\]
has order~$\ge m$.


\begin{lemma}
If $M^\bul \in \Alt(\cN)$, then $M^\bul\uim \in \Altf(\cN)$.
\end{lemma}


\begin{proof}
Let $\ua$ and~$\ub$ be nonempty words and consider the expression
$\sum\limits_{\un\UN} \shabn M^\un\uim$.
We find~$0$ if $r(\ua\,\ub)\ge m$ or if one of the letters of~$\ua$
or~$\ub$ is outside~$\cN_m$ (because, then,~$\un$ has the same property
whenever $\shabn\neq0$);
otherwise we find $\sum\limits_{\un\UN} \shabn M^\un$, which is
also~$0$ if~$M^\bul$ is supposed to be alternal.
\end{proof}


Hence $F^\bul\uim$ and~$G^\bul\uim$ are alternal and we can use
Proposition~\ref{propLieMCun}(ii) and
Proposition~\ref{propLieMCdeux}(ii) with $Y_m=G^\bul\uim\Bul$ to rewrite
$W_m = \ex^{\ad_{Y_m}} X_0 - X_0
+\ex^{\ad_{Y_m}}(I^\bul \Bul)
- F^\bul\uim \Bul$
as
\be   \label{eq:WmexmGE}
W_m = \big( \ex^{-G^\bul\uim} \times E^\bul \big) \Bul,
\qquad
E^\bul \defeq 
- \na_\lam(\ex^{G^\bul\uim}) + I^\bul\times\ex^{G^\bul\uim} 
- \ex^{G^\bul\uim} \times F^\bul\uim.
\ee
Let $C^\bul \defeq F^\bul - F^\bul\uim$, 
$\ti C^\bul \defeq G^\bul-G^\bul\uim$
and $D^\bul \defeq \ex^{G^\bul} - \ex^{G^\bul\uim}$.
Since $-\na_\lam\big(\ex^{G^\bul}\big) +  I^\bul \times \ex^{G^\bul} -
\ex^{G^\bul} \times F^\bul = 0$, we get
\be   \label{eq:EDC}
E^\bul =
\na_\lam D^\bul - I^\bul\times D^\bul
+ D^\bul\times F^\bul + \ex^{G^\bul\uim} \times C^\bul.
\ee


\begin{lemma}   \label{lem:MuimN}
(i) Suppose $M^\bul \in \kk^\UN$ and $M^\bul\uim = 0$. 
Then $M^\bul\Bul \in \cL\ugeq m$.

\noindent (ii)
Suppose $M^\bul, N^\bul \in \kk^\UN$ and $M^\bul\uim = 0$. 
Then $(M^\bul\times N^\bul)\uim = (N^\bul\times M^\bul)\uim = 0$.
\end{lemma}


\begin{proof}
Suppose $M^\bul\uim = 0$.

\noindent (i)
For any word $\un = n_1\cdots n_r$, 
$M^\un\neq0$ implies $\max\{r,\ord(B_{n_1}),\ldots,\ord(B_{n_r})\} \ge m$,
but $\ord(\Bun) \ge \max\{r,\ord(B_{n_1}),\ldots,\ord(B_{n_r})\}$,
hence  $\ord(M^\un \Bun) \ge m$ in all cases.

\noindent (ii)
Suppose $\un = n_1\cdots n_r$ with $r<m$ and $n_1,\ldots,n_r\in\cN_m$.
We have $(M^\bul\times N^\bul)^\un = \sum M^\ua N^\ub$ with summation
over all pairs of words such that $\ua\,\ub=\un$, which entails $M^\ua=0$
in each term of the sum, and similarly for $N^\bul\times M^\bul$.
\end{proof}


We have $C^\bul\uim = \ti C^\bul\uim = 0$, and 
$D^\bul = \sum_{k\ge0} \frac{1}{(k+1)!} 
\big( (G^\bul)^{\times(k+1)} - (G^\bul\uim)^{\times(k+1)} \big)$
with 
\[
(G^\bul)^{\times(k+1)} - (G^\bul\uim)^{\times(k+1)} =
\sum_{k=p+q} (G^\bul\uim)^{\times p} 
\times \ti C^\bul \times (G^\bul)^{\times q}
\quad\text{for each $k\ge0$,}
\]
whence $D^\bul\uim=0$ by Lemma~\ref{lem:MuimN}(ii).
In view of~\eqref{eq:EDC}, it follows, again by Lemma~\ref{lem:MuimN}(ii), that
$\big( \ex^{-G^\bul\uim} \times E^\bul \big)\uim = 0$,
whence $W_m \in \cL\ugeq m$ by~\eqref{eq:WmexmGE} and Lemma~\ref{lem:MuimN}(i).

\section{Resolution of the mould equation and proof of Theorem~\ref{thmB}}   \label{secpfThmB}

With the view of proving Theorem~\ref{thmB},
we now give ourselves a nonempty set~$\cN$, a field~$\kk$ of
characteristic zero and a map $\lam \col \cN \to \kk$.

Part~(i) of the statement of Theorem~\ref{thmB} requires that,
for each $A^\bul \in \Alt\RES(\cN)$,
we prove the existence and uniqueness of a pair
$(F^\bul,G^\bul) \in \Alt(\cN) \times \Alt(\cN)$ solving
\eqref{eqmouldFG}--\eqref{eqjaugeGA}.
As explained in Section~\ref{paragaSymMouldEq},
with the change of unknown mould $S^\bul \defeq \ex^{G^\bul}$,
this is equivalent to proving the existence and uniqueness of a pair
$(S^\bul,F^\bul) \in \Sym(\cN)\times\Alt(\cN)$ solving
equation~\eqref{eqmouldFS} and satisfying
\be   \label{eqAjaugeS}
\lres{ \inv S^\bul \times \na_1 S^\bul } = A^\bul.
\ee

Heuristically, here is what happens: 
it is easy to see that,
apart from the exceptional case in which $\lam(\un)\neq0$ for every
nonempty word~$\un$ (in which case $\Alt\RES(\cN) = \{0\}$ and there
is a unique solution $(S^\bul,F^\bul)$ to~\eqref{eqmouldFS} in
$\kk^\UN \times \kk^\UN$ such that $S^\est = 1$),
equation~\eqref{eqmouldFS} has in general infinitely many solutions
$(S^\bul,F^\bul) \in \kk^\UN \times \kk^\UN$ such that $S^\est = 1$
(because one is free to assign an arbitrary value to $S^\un$ whenever $\lam(\un)=0$),
but what is not obvious is the existence of at least one solution
\emph{with~$S^\bul$ symmetral and~$F^\bul$ alternal};
adding the requirement~\eqref{eqAjaugeS} removes the freedom: then we
get a unique solution $(S^\bul,F^\bul)$ in $\kk^\UN \times \kk^\UN$ such that $S^\est = 1$,
and we are left with the problem of proving that this solution is in
$\Sym(\cN)\times\Alt(\cN)$.
This will follow from the alternality of~$A^\bul$ at the price of an
excursion in the space of ``dimoulds''.


\subsection{The associative algebra of dimoulds}   \label{secexcursusdim}
$ $\\[-2.75ex]

\noindent
%
The material in this section is essentially taken from \cite{mouldSN}.


We call \emph{dimould}
any map $\UN\times\UN \to \kk$.
We denote by~$M^\bbul$ the dimould whose value on a pair of words
$(\ua,\ub)$ is $M^{\ua,\ub}$.
The set $\kUNUN$ of all dimoulds is clearly a linear space over~$\kk$,
it is also an associative $\kk$-algebra for the
\emph{dimould multiplication} $(M^\bbul,N^\bbul) \mapsto P^\bbul = M^\bbul
\times N^\bbul$ defined by a formula analogous
to~\eqref{eqdefmouldmult}:
\[
P^{\ua,\ub} \defeq 
\sum_{(\ua,\ub) = (\ua^1,\ub^1) (\ua^2,\ub^2)} M^{(\ua^1,\ub^1)} N^{(\ua^2,\ub^2)},
\]
where the concatenation in $\UN\times\UN$ is defined by
$(\ua^1,\ub^1) (\ua^2,\ub^2) = (\ua^1 \, \ua^2, \ub^1 \, \ub^2)$.

Examples of dimoulds are the \emph{decomposable dimoulds}, namely the
dimoulds of the form
\[
P^\bbul = M^\bul \otimes N^\bul,
\]
where it is meant that~$M^\bul$ and~$N^\bul$ are (ordinary) moulds and
$P^{\ua,\ub} = M^\ua N^\ub$. Note that
\be   \label{eqidmultdecdimould}
(M_1^\bul \otimes N_1^\bul) \times (M_2^\bul \otimes N_2^\bul) =
(M_1^\bul \times M_2^\bul) \otimes (N_1^\bul \times N_2^\bul) 
\ee
for any four moulds $M_1^\bul,N_1^\bul,M_2^\bul,N_2^\bul$.

Using the shuffling coefficients of Definition~\ref{defalt}, we define
a linear map
\be   \label{eqdefcoprodDe}
\De \col M^\bul \in \kk^\UN \mapsto P^\bbul = \De(M^\bul) \in \kUNUN
\ee
as follows:
\be   \label{eqdefcoprodDeuaub}
P^{\ua,\ub} \defeq \sum_{\un\in\UN} \shabn M^\un
\quad \text{for any $(\ua,\ub) \in \UN\times\UN$.}
\ee
We thus can rephrase the definition of alternality given in
Definition~\ref{defalt} and the definition of symmetrality given in~\eqref{eqdefstdsym}:
\begin{align}
&\text{\emph{A mould~$M^\bul$ is alternal if and only if
$\De(M^\bul) = M^\bul\otimes 1^\bul + 1^\bul\otimes M^\bul$.}} \\[1ex]
&\text{\emph{It is symmetral if and only if $M^\est=1$ and
$\De(M^\bul) = M^\bul \otimes M^\bul$.}}
\end{align}
It is proved in \cite[Sec.~5.2]{mouldSN} that%
\footnote{In this paper we have denoted by~$\De$ the map which was denoted by~$\tau$
in \cite{mouldSN}, because this map is essentially the coproduct of a Hopf algebra
structure that one can define and the notation~$\De$ is more common for coproducts.}
\be    \label{eqcoprDealgmorph}
\text{\emph{$\De \col \kk^\UN \to \kUNUN$ is an associative algebra morphism.}}
\ee


We end this section with an example of \emph{dimould derivation},
\ie a derivation of the dimould algebra $\kUNUN$.
%
%
\begin{lemma}    \label{lemdimderiv}
Let $\ph\col\cN\to\kk$ denote an abitrary function,
extended to~$\UN$ by~\eqref{eqextendph}.
Then the formula
\[
\wt\na_\ph \col P^\bbul \mapsto Q^\bbul, \qquad
Q^{\ua,\ub} \defeq \big( \ph(\ua)+\ph(\ub) \big) P^{\ua,\ub}
\quad \text{for all $\ua,\ub \in \UN$}
\]
defines a $\kk$-linear operator~$\wt\na_\ph$ of $\kUNUN$ which is a
dimould derivation and satisfies
\begin{align}
\label{eqtiDphMN}
& \wt\na_\ph (M^\bul\otimes N^\bul) = 
(\na_\ph M^\bul)\otimes N^\bul + M^\bul\otimes \na_\ph N^\bul \\[1ex]
\label{eqDeDtiD}
& \De(\na_\ph M^\bul) = \wt\na_\ph \De(M^\bul)
\end{align}
for any two moulds $M^\bul$ and~$N^\bul$,
where~$\na_\ph$ is the mould derivation defined by~\eqref{eqdefnaph}.
\end{lemma}

The proof of Lemma~\ref{lemdimderiv} is left to the reader
(use $\shabn \neq0 \;\Rightarrow\; \ph(\ua)+\ph(\ub) =
\ph(\un)$ for the last property).


\subsection{Proof of Part~(i) of Theorem~\ref{thmB}}
$ $\\[-2.75ex]

\noindent
%
Let $A^\bul \in \Alt\RES(\cN)$.
As explained at the beginning of Section~\ref{secpfThmB}, the strategy
is first to check the existence and uniqueness of a pair of moulds
$(S^\bul,F^\bul) \in \kk^\UN \times \kk^\UN$
satisfying~\eqref{eqmouldFS} and~\eqref{eqAjaugeS} and $S^\est=1$, 
and then to prove (with the help of dimoulds) that
$(S^\bul,F^\bul) \in \Sym(\cN)\times\Alt(\cN)$.


\parag
Let us introduce an extra unknown mould
$N^\bul = \inv S^\bul\times\na_1 S^\bul$, so that finding a solution
$(S^\bul,F^\bul)$ to~\eqref{eqmouldFS} and~\eqref{eqAjaugeS} is
equivalent to finding a solution $(S^\bul,F^\bul,N^\bul)$ to the
system of equations
\begin{align}
\label{eqnaThIThR}
\na_\lam S^\bul &= I^\bul\times S^\bul-S^\bul\times F^\bul \\[1ex]
\label{eqcLThN}
\na_1 S^\bul &= S^\bul \times N^\bul \\[1ex]
\label{eqnaRiszero}
\na_\lam F^\bul &= 0 \\[1ex]
\label{eqNresisA}
N^\bul\RES &= A^\bul.
\end{align}
The system \eqref{eqnaThIThR}--\eqref{eqNresisA}, in presence of the
condition $S^\est=1$, amounts to
$F^\est = N^\est = 0$ and, for each nonempty word~$\un$,
\begin{align}
\label{equlanaThIThR}
& \lam(\un) \, S^\un + F^\un = 
S^{`\un} - \sum_{\un=\ua\,\ub,\  \ua,\ub\neq\est}
%
%
S^\ua \, F^\ub \\[1ex]
\label{equlacLThN}
& r(\un) \, S^\un - N^\un = \sum_{\un=\ua\,\ub,\ \ua,\ub\neq\est}
%
%
S^\ua \, N^\ub \\[1ex]
\label{equlanaRiszero}
& \lam(\un)\neq0 \ens\Rightarrow\ens F^\un = 0 \\[1ex]
\label{equlaNresisA}
& \lam(\un)=0 \ens\Rightarrow\ens N^\un = A^\un
\end{align}
with~$`\un$ denoting the word~$\un$ deprived from its first letter.


We thus find a unique solution by induction on $r(\un)$:
we must take $S^\est=1$, $F^\est = N^\est = 0$ and, for $r(\un)\ge1$,
\begin{align}
\label{eqinducNRbis}
\lam(\un)\neq0 &\ens\Rightarrow\ens
\left\{ \begin{aligned}
F^\un &= 0 \\[1ex]
S^\un &= \frac{1}{\lam(\un)} \Big( S^{`\un} -
\sum_{\un=\ua\,\ub,\  \ua,\ub\neq\est}
S^\ua \, F^\ub \Big)
\\[1.5ex]
N^\un &= r(\un)\,S^\un - \sum_{\un=\ua\,\ub,\ \ua,\ub\neq\est}
S^\ua \, N^\ub
\end{aligned} \right.
\\[2ex]
\label{eqinducRESbis}
\lam(\un)=0 &\ens\Rightarrow\ens
\left\{ \begin{aligned}
F^\un &= S^{`\un} -
\sum_{\un=\ua\,\ub,\  \ua,\ub\neq\est}
S^\ua \, F^\ub \\[1ex]
N^\un &= A^\un \\[1ex]
S^\un &= \frac{1}{r(\un)} \Big( A^\un + 
\sum_{\un=\ua\,\ub,\ \ua,\ub\neq\est}
S^\ua \, N^\ub \Big).
\end{aligned} \right.
\end{align}


\parag
We now check that, in the unique solution constructed above, $S^\bul$
is symmetral and~$F^\bul$ is alternal.
Making use of the dimould formalism of Section~\ref{secexcursusdim},
and in particular of the associative algebra morphism~$\De$ defined by
\eqref{eqdefcoprodDe}--\eqref{eqdefcoprodDeuaub},
we set
\[
A^\bbul \defeq \De(A^\bul), \quad
S^\bbul \defeq \De(S^\bul), \quad
F^\bbul \defeq \De(F^\bul), \quad
N^\bbul \defeq \De(N^\bul).
\]
Our assumption amounts to 
$A^\bbul = A^\bul\otimes 1^\bul + 1^\bul\otimes A^\bul$ 
and we are to prove $S^\bbul = S^\bul\otimes S^\bul$
and $F^\bbul = F^\bul\otimes 1^\bul + 1^\bul\otimes F^\bul$.
Note that $S^{\est,\est} = S^\est = 1$.


In view of Lemma~\ref{lemdimderiv}, the dimould
derivations~$\wt\na_\lam$ and~$\wt\na_1$ are defined by
\[
\wt\na_\lam M^{\ua,\ub} \defeq \big( \lam(\ua)+\lam(\ub) \big) M^{\ua,\ub}
\quad
\text{and}
\quad
\wt\na_1 M^{\ua,\ub} \defeq \big( r(\ua)+r(\ub) \big) M^{\ua,\ub}
\quad\text{for all $\ua,\ub \in \UN$}
\]
for any dimould~$M^\bbul$.
Applying~$\De$ to each equation of the system
\eqref{eqnaThIThR}--\eqref{eqNresisA}, we get
\begin{align}
\label{eqtinaThIThR}
\wt\na_\lam S^\bbul &= \De(I^\bul)\times S^\bbul-S^\bbul\times F^\bbul \\[1ex]
\label{eqticLThN}
\wt\na_1 S^\bbul &= S^\bbul \times N^\bbul \\[1ex]
\label{eqtinaRiszero}
\wt\na_\lam F^\bbul &= 0 \\[1ex]
\label{eqtiNresisA}
N^\bbul\RES &= A^\bbul.
\end{align}
Here we have used the associative algebra morphism
property~\eqref{eqcoprDealgmorph} of~$\De$ and the
identity~\eqref{eqDeDtiD} with~$\na_\lam$ and~$\na_1$;
moreover, we have denoted by~$N^\bbul\RES$ the
resonant part of the dimould~$N^\bbul$ defined by
\[
N^{\ua,\ub}\RES \defeq \IND{\lam(\ua)+\lam(\ub)=0} \, N^{\ua,\ub} 
\qquad\text{for any $(\ua,\ub)\in\UN\times\UN$}
\]
and used the obvious identity $\big(\De(N^\bul)\big)\RES =
\De(N^\bul\RES)$
(due to the fact that $\shabn \neq0 \;\Rightarrow\; \lam(\ua)+\lam(\ub) =
\lam(\un)$).

We now observe that the system of dimould equations
\eqref{eqtinaThIThR}--\eqref{eqtiNresisA} has a unique solution
$(S^\bbul,F^\bbul,N^\bbul)$ such that $S^{\est,\est}=1$.
Indeed, these equations entail $F^{\est,\est} = N^{\est,\est} = 0$
and, by evaluating them on a pair of words $(\ua,\ub) \neq (\est,\est)$,
we get equations analogous to
\eqref{equlanaThIThR}--\eqref{equlaNresisA} which allow to determine 
$S^{\ua,\ub}$, $F^{\ua,\ub}$ and $N^{\ua,\ub}$
by induction on $r(\ua)+r(\ub)$ (distinguishing the cases $\lam(\ua)+\lam(\ub)=0$ or
$\neq0$).

Since $\De(I^\bul) = I^\bul\otimes 1^\bul + 1^\bul\otimes I^\bul$ and
$A^\bbul = A^\bul\otimes 1^\bul + 1^\bul\otimes A^\bul$,
it is easy to check directly that
$(S^\bul\otimes S^\bul, 
F^\bul\otimes 1^\bul + 1^\bul\otimes F^\bul,
N^\bul\otimes 1^\bul + 1^\bul\otimes N^\bul)$ 
is a solution of the system \eqref{eqtinaThIThR}--\eqref{eqtiNresisA}
with the initial condition $(S^\bul\otimes S^\bul)^{\est,\est}=1$
(one just has to use \eqref{eqidmultdecdimould}, \eqref{eqtiDphMN},
\eqref{eqnaThIThR}--\eqref{eqNresisA} and the identities
$(N^\bul\otimes 1^\bul)\RES = N^\bul\RES\otimes 1^\bul$,
$(1^\bul\otimes N^\bul)\RES = 1^\bul\otimes N^\bul\RES$).

The uniqueness of the solution of the system of dimould equations implies
\[
(S^\bbul,F^\bbul,N^\bbul)
= (S^\bul\otimes S^\bul, 
F^\bul\otimes 1^\bul + 1^\bul\otimes F^\bul,
N^\bul\otimes 1^\bul + 1^\bul\otimes N^\bul)
\]
in particular~$S^\bul$ is symmetral and~$F^\bul$ is alternal.


\parag
The induction formulas \eqref{eqinducNRbis}--\eqref{eqinducRESbis}
that we have obtained for~$F^\bul$ and~$S^\bul$
coincide with
\eqref{eqinducNR}--\eqref{eqinducRES}.
Setting $G^\bul = \log S^\bul$, we get an alternal mould, inductively
determined by~\eqref{eqGlogS}.

This ends the proof of Part~(i) of Theorem~\ref{thmB}.


\subsection{Proof of Part~(ii) of Theorem~\ref{thmB}}
$ $\\[-4ex]

%
\parag
Recall that the mould exponential
$G^\bul \mapsto S^\bul = \ex^{G^\bul}$ is a bijection between the set
of all moulds~$G^\bul$ such that $G^\est = 0$ and the set of all
moulds~$S^\bul$ such that $S^\est = 1$, which induces a bijection
$\Alt(\cN)\to\Sym(\cN)$.
There is thus a bijection between the solutions
$(F^\bul,G^\bul) \in \kk^\UN\times\kk^\UN$ to
equation~\eqref{eqmouldFG} such that $G^\est = 0$ and the solutions
$(F^\bul,S^\bul) \in \kk^\UN\times\kk^\UN$ to
equation~\eqref{eqmouldFS} such that $S^\est = 1$.
We rewrite equation~\eqref{eqmouldFS} as
\begin{align}
\label{eqFintermsofS}
& F^\bul = 
  \inv S^\bul \times I^\bul \times S^\bul 
  - \inv S^\bul \times \na_\lam S^\bul, \\[1ex]
\label{eqnalamF}
& \na_\lam F^\bul = 0.
\end{align}


Starting with a solution $(F^\bul,G^\bul) \in \Alt(\cN) \times \Alt(\cN)$
to~\eqref{eqmouldFG} and setting
$S^\bul \defeq \ex^{G^\bul} \in \Sym(\cN)$,
we get a solution $(F^\bul,S^\bul) \in \Alt(\cN) \times \Sym(\cN)$
to \eqref{eqFintermsofS}--\eqref{eqnalamF};
using the change $K^\bul = \ex^{J^\bul}$ 
(as in Section~\ref{paragaJaugeK}), 
we are asked to prove that the map
\be   \label{eqdeftiFtiS}
K^\bul \mapsto (\ti F^\bul, \ti S^\bul) = 
\big( \inv K^\bul \times F^\bul \times K^\bul,
\, S^\bul \times K^\bul \big) 
\ee
establishes a one-to-one correspondence between $\Sym\RES(\cN)$ and the set of all solutions
$(\ti F^\bul, \ti S^\bul) \in \Alt(\cN)\times\Sym(\cN)$ to \eqref{eqFintermsofS}--\eqref{eqnalamF},
and that
\be   \label{eqvarjauge}
\lres{ \inv \ti S^\bul \times \na_1 \ti S^\bul } = 
\inv K^\bul \times \JJ(G^\bul) \times K^\bul
+ \inv K^\bul \times \na_1 K.
\ee


\parag
Suppose that $K^\bul \in \Sym\RES(\cN)$ and define $(\ti F^\bul, \ti
S^\bul)$ by~\eqref{eqdeftiFtiS}.
Since $\ti S^\bul = S^\bul \times K^\bul$, this mould is symmetral
(recall that $( \Sym(\cN), \times )$ is a group---see \eg \cite[Prop.~5.1]{mouldSN});
since $\na_\lam$ is a derivation which annihilates~$K^\bul$, we have 
$\na_\lam \ti S^\bul = (\na_\lam S^\bul)\times K^\bul$ and
\[
\inv \ti S^\bul \times I^\bul \times \ti S^\bul - \inv \ti S \times\na_\lam\ti S
= \inv K^\bul \times \big(
\inv S^\bul \times I^\bul \times  S^\bul - \inv S \times\na_\lam S
\big) \times K^\bul,
\]
which, by~\eqref{eqFintermsofS}, is $\inv K^\bul \times F^\bul
\times K^\bul = \ti F^\bul$.
Thus, $(\ti F^\bul, \ti S^\bul)$ satisfies~\eqref{eqFintermsofS}.

On the other hand, by~\eqref{defSymexpAlt} and Proposition~\ref{propLieMCun}(ii), 
$\ti F = \inv K^\bul \times F^\bul \times K^\bul$ is alternal. 
It is easy to check that~$\ti F^\bul$ satisfies~\eqref{eqnalamF}
because~$F^\bul$ satisfies~\eqref{eqnalamF}:
$0 = \inv K^\bul \times \na_\lam F^\bul \times K^\bul = \na_\lam\ti F$.
It is so because $\na_\lam$ is derivation which annihilates both~$K^\bul$ and~$\inv K^\bul$;
the fact that also~$\inv K^\bul$ is $\lam$-resonant (\ie $\na_\lam \inv K^\bul=0$) is an elementary
property of $\lam$-resonant moulds, which is part of
\begin{lemma}   \label{lemelemptyres}
Suppose that $M^\bul$ is a $\lam$-resonant mould. Then also $\na_1
M^\bul$ is $\lam$-resonant, and
\[
\lres{ M^\bul\times N^\bul } = M^\bul \times N^\bul\RES, \quad
%
%
\lres{ N^\bul\times M^\bul } = N^\bul\RES \times M^\bul
\qquad \text{for any mould~$N^\bul$.}
\]
If moreover~$M^\bul$ is invertible, then also $\inv M^\bul$ is $\lam$-resonant.
\end{lemma}

The proof of Lemma~\ref{lemelemptyres} is left to the reader.
\smallskip

We now compute the gauge generator of $\log \ti S^\bul$:
by Lemma~\ref{lemelemptyres}, the $\lam$-resonant part of
\[
\inv \ti S^\bul \times \na_1 \ti S^\bul = 
\inv K^\bul \times \inv S^\bul \times \big(
(\na_1 S^\bul)\times K^\bul + S^\bul \times \na_1 K^\bul \big)
\]
is 
$\inv K^\bul\times\lres{ \inv S^\bul \times \na_1 S^\bul }\times K^\bul
+ \inv K^\bul\times\na_1 K^\bul
= \inv K^\bul \times \JJ(G^\bul) \times K^\bul
+ \inv K^\bul \times \na_1 K^\bul$.
This is~\eqref{eqvarjauge}.


\parag
Conversely, suppose that
$(\ti F^\bul, \ti S^\bul) \in \Alt(\cN)\times\Sym(\cN)$ is a solution
to \eqref{eqFintermsofS}--\eqref{eqnalamF}.
We define $K^\bul \defeq \inv S^\bul \times \ti S^\bul \in \Sym(\cN)$.
Inserting
\be   \label{eq:tiSSK}
\ti S^\bul = S^\bul \times K^\bul
\ee
in $\ti F^\bul =   \inv \ti S^\bul \times I^\bul \times \ti S^\bul 
  - \inv \ti S^\bul \times \na_\lam \ti S^\bul$,
we get
\be   \label{eq:tiFinvKFK}
\ti F^\bul = \inv K^\bul \times \big(
\inv S^\bul \times I^\bul \times S^\bul \times K^\bul
- \inv S^\bul \times \na_\lam(S^\bul \times K^\bul) \big)
= \inv K^\bul \times ( F^\bul \times K^\bul - \na_\lam K^\bul ),
\ee
\ie $\na_\lam K^\bul = F^\bul \times K^\bul - K^\bul \times \ti F^\bul$.
We are in a position to apply


\begin{lemma}   \label{lem:implyPres}
Suppose that $M^\bul, N^\bul, P^\bul \in \kk^\UN$, $M^\est=N^\est=0$,
$M^\bul$ and $N^\bul$ are $\lam$-resonant and
\be   \label{eq:naPMPPN}
\na_\lam P^\bul = M^\bul \times P^\bul - P^\bul \times N^\bul .
\ee
Then $P^\bul$ is $\lam$-resonant.
\end{lemma}


Taking Lemma~\ref{lem:implyPres} for granted, we thus obtain that
$K^\bul$ is $\lam$-resonant, hence $K^\bul \in \Sym\RES(\cN)$, 
and~\eqref{eq:tiFinvKFK} yields $\ti F = \inv K^\bul \times F^\bul \times K^\bul$,
which together with~\eqref{eq:tiSSK} gives $(\ti F^\bul, \ti S^\bul)$
as the image of~$K^\bul$ by the map~\eqref{eqdeftiFtiS}.
The proof of Theorem~\ref{thmB}(ii) is then complete.


\begin{proof}[Proof of Lemma~\ref{lem:implyPres}]
Let us show that
\be   \label{eq:naPunzer}
\lam(\un) P^\un = 0
\ee
for every $\un\in\UN$ by induction on~$r(\un)$.
The property holds for $\un=\est$ or, more generally, for
$\lam(\un)=0$, we thus suppose that $\un\in\UN$ has $r(\un)\ge1$ and $\lam(\un)\neq0$,
and that~\eqref{eq:naPunzer} holds for all words of length $< r(\un)$.
It follows from~\eqref{eq:naPMPPN} that
\be   \label{eq:lamunPun}
\lam(\un) P^\un = \sum_{\un = \ua\,\ub} (M^\ua P^\ub - P^\ua N^\ub)
= \sst\sum_{\un=\ua\,\ub} (M^\ua P^\ub - P^\ua N^\ub),
\ee
where the symbol $\sst\sum$ indicates that we can restrict the
summation to non-trivial decompositions
(it is so because $M^\un = N^\un = 0$, since $\lam(\un)\neq0$, and
$M^\est = N^\est = 0$).
But, in the \rhs\ of~\eqref{eq:lamunPun}, each term between
parentheses vanishes, because either
$\lam(\ua)\neq0$ and $M^\ua=P^\ua=0$ (by the assumption
on~$M^\ua$ and the inductive hypothesis),
or $\lam(\ua)=0$, but then $\lam(\ub)\neq0$ and $M^\ub=P^\ub=0$ (for similar reasons).
\end{proof}


\centerline{\Large\sc Five dynamical applications}
\addcontentsline{toc}{part}{\vspace{-2.1em} \\ \sc Five dynamical applications\vspace{.4em}}

\vspace{.5cm}




%

We now turn to examples of application of Theorem~\ref{thmA}.
The Lie algebras in these examples will consist of 
%
vector fields with their natural Lie brackets $\vf{\cdot\,}\cdot$
%
%
or, in presence of a symplectic structure, 
Hamiltonian functions with the Lie bracket
$\ham{\cdot\,}\cdot \defeq \{\cdot\,,\cdot\}$ (Poisson bracket)
%
%
or, in the quantum case, operators of a Hilbert space
with the Lie bracket
$\qu{\cdot\,}\cdot \defeq \frac{1}{\I\hb}\times\text{commutator}$.
We will deal with formal objects (\ie defined by means of formal
series, either in the dynamical variables or in some external
parameter), and this gives rise to a natural Lie algebra filtration.


\section{Poincar\'e-Dulac normal forms}   \label{secPDNF}


\parage
Let $N\in\N^*$. A formal vector field is the same thing as a
derivation of the algebra of formal series $\C[[z_1,\ldots,z_N]]$ and
is of the form
\[
X = \sum\limits_{j=1}^N v_j(z_1,\ldots,z_N) \pa_{z_j}.
\]
We take $\kk \defeq \C$ and $\cL \defeq $ the Lie algebra of formal
vector fields whose components~$v_j$ have no constant term.
We get a complete filtered algebra by setting $X \in \cL\ugeq m$
if its components~$v_j$, as formal series, have order~$\ge m+1$.

Let $X\in\cL$.
The \emph{formal normalization problem} consists in finding a formal change
of variables which simplifies the expression of~$X$ as much as
possible.
We assume that~$X$ has a diagonal linear part:
\[
X_0 =
\sum_{j=1}^N \om_j z_j \pa_{z_j}
\]
with ``spectrum vector''
$\om = (\om_1,\ldots,\om_N) \in \C^N$.
The components of $B\defeq X-X_0$ have order~$\ge2$, hence,
introducing
\[
\cM \defeq \big\{\, 
(j,k) \in \{1,\ldots,N\} \times \N^N \mid \abs{k}\ge2
\,\big\},
\]
we can write the expansion of~$X-X_0$ as
$B = \sum\limits_{(j,k)\in\cM} b_{j,k} z^k \pa_{z_j}$
with coefficients $b_{j,k} \in \C$.
It turns out that the monomial vector fields $z^k \pa_{z_j}$ are
eigenvectors of $\ad_{X_0}$:
\be   \label{eq:monomev}
\vf{X_0}{z^k \pa_{z_j}} = \big( \scal{k}{\om} - \om_j \big) z^k
\pa_{z_j}
\quad\text{for each $(j,k)\in\cM$}
\ee
(where $\scal{\cdot\,}\cdot$ denotes the standard scalar product),
we thus set
\begin{align*}
\cN &\defeq \big\{\, \scal{k}{\om} - \om_j \mid (j,k)\in\cM 
\;\text{and}\;
b_{j,k} \neq 0 \,\big\} \subset \C,
\\[1ex]
B_\lam &\defeq \sum_{ \substack{(j,k)\in\cM \,\text{such that}\\[.5ex]
    \scal{k}{\om} - \om_j=\lam} }
b_{j,k} z^k \pa_{z_j}
\quad\text{for each $\lam\in\cN$,}
\end{align*}
so that $X=X_0+\sum\limits_{\lam\in\cN}B_\lam$ 
and $\vf{X_0}{B_\lam} = \lam B_\lam$ for each $\lam\in\cN$.


\parage
Let us apply Theorem~\ref{thmA}:
with each choice of $A^\bul\in\Alt_0(\cN)$ is associated a pair of
alternal moulds, $F^\bul$ and~$G^\bul$ explicitly given by
\eqref{eqinducinit}--\eqref{eqGlogS}, which give rise to formal vector
fields~$Z$ and~$Y$ such that~\eqref{eqA1} holds:
the automorphism $\ex^{\ad_Y}$ of~$\cL$ maps $X=X_0+B$ to $X_0+Z$ and
$\vf{X_0}{Z} = 0$.
Moreover,~$Z$ and~$Y$ are explicitly given by the expansions
\eqref{eqmouldexpZ}--\eqref{eqmouldexpY} (with the convention of
Definition~\ref{def:incluscase}: the map~$\lam$ is to be interpreted
as the inclusion map $\cN \hookrightarrow \C$).


In this context, a formal vector field which commutes with~$X_0$ is
called ``resonant''.
According to~\eqref{eq:monomev}, this means that it is a sum of
``resonant monomials'', \ie multiples of elementary vector fields of the form
$z^k\pa_{z_j}$ with 
\be   \label{eq:jkresonant}
(j,k) \in \cM \quad\text{such that}\quad
\scal{k}{\om} - \om_j=0.
\ee
It may happen that there exist no resonant monomial at all:
one says that the spectrum vector~$\om$ is ``non-resonant'' if
equation~\eqref{eq:jkresonant} has no solution
(a kind of arithmetical condition).
Necessarily $Z=0$ in that case (although $F^\bul$ might be nonzero).


The first part of~\eqref{eqA1} thus says that~$Z$ is a formal resonant
vector field;
classically, $X_0+Z$ is called a \emph{Poincar\'e-Dulac normal form}.
In \cite{EV95}, the particular Poincar\'e-Dulac normal form
corresponding to the choice $A^\bul=0$ (zero gauge solution of
equation~\eqref{eqmouldFG}) is called ``regal prenormal form''.


The automorphism $\ex^{\ad_Y}$ of~$\cL$ is nothing but the action of
the formal flow~$\Phi$ of~$Y$ at time~$1$ by pull-back:
$\ex^{\ad_Y} X = \Phi_*\ii X$, 
hence the second part of~\eqref{eqA1} says that
$\Phi_*\ii X = X_0+Z$,
which corresponds to the formal change of coordinates
$z\mapsto \Phi\ii(z)$ obtained by flowing at time~$1$ along~$-Y$.

%
We have thus recovered the classical results by Poincar\'e and Dulac,
according to which \emph{one can formally conjugate~$X$ to its linear
part~$X_0$ when~$\om$ is non-resonant and, in the general case, to a
formal vector field the expression of which contains only resonant
monomials}.

It is well known that, in general, there is more than one Poincar\'e-Dulac
normal form.


\parage
For a resonant vector~$\om$, there may be only one resonance
relation~\eqref{eq:jkresonant} (\eg for $\om=(2,1)$ in dimension
$N=2$) or infinitely many of them (\eg for $\om=(-1,1)$).
A generic vector $\om$ in~$\C^N$ is non-resonant, but for certain
classes of vector fields like the class of Hamiltonian vector fields
the spectrum vector is necessarily resonant---see Section~\ref{secCBNF}.

As already mentioned, when $\om$ is non-resonant, $F^\bul$ is not
necessarily trivial.
This is because the alphabet $\cN \subset \C^*$ is not necessarily
stable under addition and it may happen that there is a nonempty word
$\ula = \lam_1\cdots\lam_r \in \UN$ such that $\lam_1+\cdots+\lam_r=0$,
in which case formula~\eqref{eqfacileS} fails to define the value of~$S^\ula$.
In fact, in that case, there is no non-trivial mould~$S^\bul$ such
that $\na S^\bul = I^\bul \times S^\bul$.
However, we repeat that \emph{Poincar\'e's formal linearization
  theorem holds in that situation}: we necessarily have $\Bula=0$ for
such a word~$\ula$, and $Z=0$, since there are no non-trivial resonant
formal vector fields.

Here is an example in dimension $N=2$: the spectrum vector $\om =
(5\varpi,2\varpi)$ with $\varpi\in\R^*$ is non-resonant but if we assume
that, associated with $(j,k)=\big( 1, (0,2) \big)$ or $\big( 1,
(0,3) \big)$, there are nonzero coefficients $b_{j,k}$, then $\cN$
contains $\lam=-\varpi$ and~$\mu=\varpi$ and
\eqref{eqinducinit}--\eqref{eqGlogS} yield $F^{\lam\mu} = \frac{1}{\varpi} =
-F^{\mu\lam}$ and $S^{\lam\mu} = -\frac{1}{2\varpi^2} = S^{\mu\lam}$.


\begin{remark}   \label{remStrongNR}
If $\om\in\C^N$ is ``strongly non-resonant'' in the sense that
\[
\scal{k}{\om}\neq0
\quad\text{for any nonzero $k\in\Z^N$,}
\]
then the sum of the letters is nonzero for every nonempty word, hence
$F^\bul=0$ and the symmetral mould~$S^\bul$ is entirely determined by
the utterly simple formula~\eqref{eqfacileS}.
So, in \emph{that} case, the mould equation $\na S^\bul = I^\bul \times
S^\bul$ has a symmetral solution, which is sufficient to obtain formal linearization by mould calculus.
\end{remark}


\begin{remark}
On the other hand, it may happen that $\om$ is resonant but~$0$ does
not belong to the additive monoid generated by~$\cN$ (in particular
this requires that $b_{j,k}=0$ for every $(j,k)\in\cM$ such that $\scal{k}{\om} - \om_j=0$).
In that case $F^\bul$ is necessarily~$0$, hence $X$ is formally linearizable.
\end{remark}


\parage
The formal flow~$\Phi$ can be directly computed in terms of
the symmetral mould $S^\bul = \ex^{G^\bul}$:
viewing the $B_\lam$'s as differential operators which can be composed
(and not only Lie-bracketed), we can define the associative comould
$\ula=\lam_1\cdots\lam_r \in \UN \mapsto B_{\lam_1\cdots\lam_r} = B_{\lam_r}\cdots B_{\lam_1}$ and,
according to the end of Remark~\ref{remdefsym}, we get
\[
Y = \sum_{r\geq 1} \;\sum_{\lam_1,\ldots,\lam_r\in\cN}\,
G^{\lam_1\cdots\lam_r} B_{\lam_1\cdots\lam_r}
\]
(in general $B_{\lam_1\cdots\lam_r} \notin \cL$, but the above sum
is in~$\cL$ and coincides with~$Y$), and
\[
\ex^Y = \ID +
\sum_{r\geq 1} \;\sum_{\lam_1,\ldots,\lam_r\in\cN}\,
S^{\lam_1\cdots\lam_r} B_{\lam_1\cdots\lam_r}
\]
(this operator is not in~$\cL$). 
Now $\ex^{Y} f = f\circ\Phi$ for any $f\in \C[[z_1,\ldots,z_N]]$, hence
$\Phi = (\Phi_1,\ldots,\Phi_N)$ with 
\[
\Phi_j(z_1,\ldots,z_N) = z_j +
\sum_{r\geq 1} \;\sum_{\lam_1,\ldots,\lam_r\in\cN}\,
S^{\lam_1\cdots\lam_r} B_{\lam_1\cdots\lam_r} z_j
\quad\text{for $j=1,\ldots,N$.}
\]
There is a similar formula for~$\Phi\ii$ involving~$\inv S^\bul$. 
%


\section{Classical Birkhoff normal forms}    \label{secCBNF}


\parage
Let $d\in\N^*$. We now set
\[
\cL^\kk \defeq \big\{\, f \in \kk[[x_1,\ldots,x_d,y_1,\ldots,y_d]] \mid
\text{$f$ has order $\ge 2$} \,\big\},
\qquad
\text{$\kk = \R$ or $\C$.}
\]
The symplectic form 
$\sum_{j=1}^d \dd x_j \wedge \dd y_j$
induces the Poisson bracket
$\{f,g\} \defeq \sum_{j=1}^d \big(
\frac{\pa f\,}{\pa x_j}\frac{\pa g\,}{\pa y_j} - \frac{\pa f\,}{\pa y_j}\frac{\pa g\,}{\pa x_j}
\big)$,
which makes~$\cL^\kk$ a Poisson algebra over~$\kk$, and thus a Lie
algebra over~$\kk$ with 
$\ham{\cdot\,}\cdot \defeq \{\cdot\,,\cdot\}$.
We get a complete filtered Lie algebra by setting $X\in
\cL^\kk\ugeq m$ if, as a power series, it has order~$\ge m+2$.

Any $X\in\cL^\kk$ generates a formal Hamiltonian vector field, namely 
\[
\{X,\cdot\,\} = 
\sum_{j=1}^d \Big(
\frac{\pa X}{\pa x_j}\frac{\pa\,\;}{\pa y_j} - \frac{\pa X}{\pa y_j}\frac{\pa\,\;}{\pa x_j}
\Big)
\]
viewed as a derivation of the associative algebra
$\kk[[x_1,\ldots,x_d,y_1,\ldots,y_d]]$.
Let~$X_0$ be the quadratic part of~$X$, so that $\{X_0,\cdot\,\}$ is the
linear part of the formal vector field $\{X,\cdot\,\}$.
The corresponding matrix is Hamiltonian, hence its eigenvalues come
into pairs of opposite complex numbers and we cannot avoid resonances
in this case.
From now on, we assume that 
\be   \label{eqdefXzeromjxjyj}
X_0 = \sum_{j=1}^d \dem \om_j (x_j^2+y_j^2),
\quad \text{hence}\ens
\{X_0,\cdot\,\} = 
\sum_{j=1}^d \om_j \Big(
x_j \frac{\pa\,\;}{\pa y_j} - y_j \frac{\pa\,\;}{\pa x_j}
\Big),
\ee
with a ``frequency vector'' $\om = (\om_1,\ldots,\om_d)\in\kk^d$,
so the eigenvalues of the linear part of the vector field are
$\I\,\om_1,\ldots,\I\,\om_d,-\I\,\om_1,\ldots,-\I\,\om_d$
(which corresponds to a totally elliptic equilibrium point at the
origin when $\kk=\R$).


The formal Hamiltonian normalization problem consists in finding a
formal symplectomorphism~$\Phi$ such that the expression of
$X\circ\Phi$ is as simple as possible
(so that the expression of the conjugate Hamiltonian vector field
$\Phi\ii_* \{X,\cdot\,\}$ is as simple as possible).
We will apply Theorem~\ref{thmA} in the Lie algebra~$\cL^\C$ of
complex formal Hamiltonian functions so as to recover the classical
result according to which 
\begin{quote}{\emph{there exists a formal
symplectomorphism~$\Phi$ (with real coefficients if $\kk=\R$) such
that $X\circ\Phi$ Poisson-commutes with~$X_0$,}}\end{quote} 
\ie $X\circ\Phi$ is a \emph{Birkhoff normal form}
(which implies, at the level of vector fields, that
$\Phi\ii_* \{X,\cdot\,\}$ is a Hamiltonian Poincar\'e-Dulac normal
form).


\parage   \label{secconjlingC}
The series
\be   \label{eq:defzjwj}
z_j(x,y) \defeq \tfrac{1}{\sqrt2} (x_j + \I\, y_j),
\quad
w_j(x,y) \defeq \tfrac{1}{\sqrt2} (\I\,x_j + y_j),
\qquad j = 1, \ldots, d,
\ee
satisfy 
$\sum \dd x_j \wedge \dd y_j = \sum \dd z_j \wedge \dd w_j$
and 
\be   \label{eq:Poissmonomev}
\{ X_0, z^k w^\ell \} = \I \, \scal{k-\ell}{\om} \, z^k w^\ell
\quad\text{for any $k,\ell\in\N^d$.}
\ee
Using them as a change of coordinates and writing the
generic formal series as
\[
\sum_{k,\ell\in\N^d} b_{k,\ell} \, x^k y^\ell 
= \sum_{k,\ell\in\N^d} c_{k,\ell} \, z^k w^\ell,
\]
we identify the complex Poisson algebras
$\C[[x_1,\ldots,x_d,y_1,\ldots,y_d]]$ and
$\C[[z_1,\ldots,z_d,w_1,\ldots,w_d]]$.
%
%
The real Poisson algebra $\R[[x_1,\ldots,x_d,y_1,\ldots,y_d]]$ can be
seen as the subspace consisting of the fixed points of the
conjugate-linear involution~$\gC$ which maps
%
%
$\sum b_{k,\ell} \, x^k y^\ell$ to $\sum \ov{b_{k,\ell}} \, x^\ell
y^k$;
note that~$\gC$ maps $\sum c_{k,\ell} \, z^k w^\ell$ to
$\sum (-\I)^{\abs{k+\ell}} \ov{c_{k,\ell}} \, z^\ell w^k$,
hence the coefficients $b_{k,\ell}$ are real if and only if 
\be   \label{eqcondreel}
\ov{c_{k,\ell}} = \I^{\abs{k+\ell}} c_{\ell,k}
\quad\text{for all $k,\ell\in\N^d$.}
\ee

Let $X\in\cL^\kk$ with quadratic part~$X_0$ as in~\eqref{eqdefXzeromjxjyj}.
Introducing 
\[
\cM \defeq \big\{\, (k,\ell)\in\N^d\times\N^d \mid \abs{k} + \abs{\ell} \ge
3 \,\big\},
\]
we can decompose $B \defeq X-X_0 \in \cL^\kk_1$ as
$B = \sum\limits_{(k,\ell)\in\cM} c_{k,\ell} \, z^k w^\ell$ with
coefficients $c_{k,\ell} \in \C$, and set
\be   \label{eq:defBndefcNNd}
B_n \defeq \sum_{ \substack{(k,\ell)\in\cM \,\text{such}\\[.5ex]
    \text{that}\, k-\ell=n} } 
c_{k,\ell} \, z^k w^\ell \in \cL^\C_1
\qquad\text{for $n\in \cN \defeq \Z^d$,}
\ee
so that $X = X_0 + \sum B_n$ and, for each $n\in\cN$,
\be   \label{eq:deflamHam}
\{ X_0, B_n \} = \lam(n) B_n, \qquad
\lam(n) = \I \, \scal{n}{\om} \in \C.
\ee
Moreover, if $\kk=\R$, then condition~\eqref{eqcondreel}
holds, whence 
\be   \label{eqcondreelB}
\gC(B_n)=B_{-n}
\quad\text{for all $n\in\Z^d$.}
\ee
in that case.


\parage   \label{sec:argumreal}
Let us apply Theorem~\ref{thmA} to~$\cL^\C$.
For any complex-valued $A^\bul\in\Alt\RES(\cN)$ (recall that
$\cN=\Z^d$ and~$\lam$ is defined by~\eqref{eq:deflamHam}), 
Theorem~\ref{thmB} yields alternal moulds $F^\bul, G^\bul \in \C^\UN$, explicitly given by
\eqref{eqinducinit}--\eqref{eqGlogS}, such that
$Z, Y \in \cL\ugeq1^\C$ defined by
\[
Z = \sum_{r\ge1} \, \sum_{\un\in\cN^r} \,
\frac{1}{r} F^\un \, \Bun,
\qquad
Y = \sum_{r\ge1} \, \sum_{\un\in\cN^r} \,
\frac{1}{r} G^\un \, \Bun
\]
satisfy~\eqref{eqA1}.

Formulas \eqref{eqinducinit}--\eqref{eqGlogS} show that, if $\kk=\R$ and $A^\bul$ is
real-valued\footnote{   \label{footA}
In fact it is sufficient that the complex conjugate of $A^{n_1\cdots n_r}$ is
$A^{(-n_1)\cdots (-n_r)}$ for any word ${n_1\cdots n_r}$.
}, then the complex conjugate of $F^{n_1\cdots n_r}$ is
$F^{(-n_1)\cdots (-n_r)}$ and similarly for~$G^\bul$ (because
$\ov{\lam(n)} = \lam(-n)$ for each $n\in\cN$);
on the other hand, $\gC$ maps
$\Bun = \{B_{n_r},\{\ldots\{B_{n_2},B_{n_1}\}\ldots\}\}$
to
$\{B_{-n_r},\{\ldots\{B_{-n_2},B_{-n_1}\}\ldots\}\}$
%
%
(because of~\eqref{eqcondreelB} and because $\gC$ is a 
real Lie algebra automorphism\footnote{
%
%
Indeed, $\gC$ can be viewed as the symmetry 
$f_1+\I f_2 \mapsto f_1 -\I f_2$
associated with the direct sum 
$\cL^\C = \cL^\R \oplus \I \cL^\R$,
it is a real Lie algebra automorphism because~$\cL^\R$ is a real Lie subalgebra.
} of~$\cL^\C$)
and is conjugate-linear,
hence we get 
%
%
$Z, Y \in \cL\ugeq1^\R$
%
%
in that case.


So $Z,Y\in\cL\ugeq1^\kk$ whether $\kk=\C$ or~$\R$.
The automorphism $\ex^{\ad_Y}$ of~$\cL^\kk$ is nothing but the action
of the formal flow~$\Phi$ at time~$1$ of the formal Hamiltonian vector
field $\{Y,\cdot\,\}$ by composition:
$\ex^{\ad_Y} X = X\circ\Phi$, 
hence the second part of~\eqref{eqA1} says that
$X\circ\Phi = X_0+Z$,
where~$\Phi$ is a formal symplectomorphism with coefficients in~$\kk$,
which implies
$\Phi\ii_* \{X,\cdot\,\} = \{X_0+Z,\cdot\,\}$
at the level of the formal Hamiltonian vector fields.
The components of~$\Phi$ can be directly computed from the symmetral
mould~$S^\bul$ by means of~\eqref{eqexpadYSB}:
\begin{align*}
\Phi_j(x,y) = x_j +
\sum_{r\geq 1} \;\sum_{n_1,\ldots,n_r\in\cN}\,
S^{n_1\cdots n_r} \ad_{B_{n_r}} \cdots \ad_{B_{n_1}} x_j \\[1ex]
\Phi_{d+j}(x,y) = y_j +
\sum_{r\geq 1} \;\sum_{n_1,\ldots,n_r\in\cN}\,
S^{n_1\cdots n_r} \ad_{B_{n_r}} \cdots \ad_{B_{n_1}} y_j
\end{align*}
for $j=1,\ldots,N$
(the series~$x_j$ and~$y_j$ have been excluded from the definition
of~$\cL^\kk$, but~\eqref{eqexpadYSB} holds as an identity between
operators acting in the whole of
$\kk[[x_1,\ldots,x_d,y_1,\ldots,y_d]]$).


The first part of~\eqref{eqA1} says that $X_0+Z$ is a ``Birkhoff
  normal form'', in the sense that it Poisson-commutes with~$X_0$.
According to~\eqref{eq:Poissmonomev}, this means that all the monomials in its $(z,w)$-expansion are of the
form $c_{k,\ell}\, z^k w^\ell$ with $\scal{k-\ell}{\om}=0$.


\parage
Instead of~\eqref{eq:defBndefcNNd}, one can as well take
\[
\cN \defeq \{\, \I \, \scal{k-\ell}{\om} \mid (k,\ell)\in\cM
\;\text{and}\;
c_{k,\ell} \neq 0 \,\} \subset \C,
\qquad
B_\lam \defeq \sum_{ \substack{(k,\ell)\in\cM \,\text{such that}\\[.5ex]
    \I \, \scal{k-\ell}{\om} = \lam} }
c_{k,\ell}  \, z^k w^\ell,
\]
so that~\eqref{eq:deflamHam} is replaced by $\{ X_0,B_\lam \} = \lam
B_\lam$ for each $\lam\in\cN$ and one can use the formalism of Definition~\ref{def:incluscase}.


%
When~$\om$ is strongly non-resonant in the sense of
Remark~\ref{remStrongNR}, 
the relation $\scal{k-\ell}{\om}=0$ implies $k-\ell=0$, hence
\[
Z = \sum_{ \abs{\ell} \ge 2 } C_\ell \, z^\ell w^\ell 
= \sum_{ \abs{\ell} \ge 2 } \I^{\abs{\ell}} C_\ell \, I_1^{\ell_1}\cdots I_d^{\ell_d},
\qquad
I_j \defeq \dem(x_j^2+y_j^2)
\quad\text{for $j=1,\ldots,d$,}
\]
with certain complex coefficients~$C_\ell$, which satisfy
$\I^{\abs{\ell}} C_\ell \in \R$ when $\kk=\R$.


It is easy to check that, when $\om$ is strongly non-resonant, the
Birkhoff normal form is unique (but not the formal symplectomorphism
conjugating~$X$ to it).


\parage
%
%
\hspace{-.4em}\emph{Remark.}\label{remBNFpoleps}
  Exactly the same formalism would apply to the perturbative situation of
  a Hamiltonian~$X$ which is also a formal series in~$\eps$ (an
  indeterminate playing the role of a parameter).
  We would take
  $\cL^\kk \defeq \kk[[x_1,\ldots,x_d,y_1,\ldots,y_d,\eps]]$ with
  $\kk = \R$ or~$\C$, with Lie bracket
  $\ham{\cdot\,}\cdot \defeq \{\cdot\,,\cdot\}$ as before, and with filtration
  induced by the total order in the $2d+1$ indeterminates.
Then, for any $X = X_0 + B$ with~$X_0$ as in~\eqref{eqdefXzeromjxjyj}
and $B \in \cL^\kk\ugeq1$,
Theorem~\ref{thmA} yields a formal symplectomorphism~$\Phi$ such that
$X\circ\Phi=X_0+Z$ Poisson-commutes with~$X_0$.
%


\parage   \label{sec:BNFclassmixte}
The above formalism, as it stands, does not allow us to deal directly
with $C^\infty$ functions of $(x,y)$, but there is a simple variant
which allows for mixed Hamiltonians, formal in~$\eps$ (as in
Remark~\ref{remBNFpoleps}) with coefficients $C^\infty$ in
$(x,y)$.
However, to have a decomposition of $X-X_0$ as a formally summable
series of eigenvectors of $\{X_0,\cdot\,\}$, we must restrict
ourselves to a certain kind of $C^\infty$ functions.
With a view to allowing for comparison with certain quantum
Hamiltonians in Section~\ref{secSemiCl}, we denote by~$\cS$ the
Schwartz class and set, for $\kk = \R$ or $\C$,
\begin{align*}
\gS_0^\kk &\defeq \Big\{\, f \in \cS(\R^d\times\R^d,\kk) \mid 
\exists \ti f\in C^\infty\big( (\R_{\ge0})^d,\kk \big) 
\;\text{such that}\; f(x,y) \equiv \ti f\big(
\tfrac{x_1^2+y_1^2}{2},\ldots,\tfrac{x_d^2+y_d^2}{2}
\big) \,\Big\},
\\[1ex]
\gS^\kk &\defeq 
\Big\{\, \sum_{(k,\ell)\in \Om}
b_{k,\ell}(x,y) \, x^k y^\ell \mid
\text{$\Om$ finite subset of $\N^d\times\N^d$,}
\;\text{$b_{k,\ell}\in \gS_0^\kk$ for each $(k,\ell)\in \Om$}
\,\Big\},
\\[1ex]
\cL^\kk &\defeq \gS^\kk[[\eps]].
\end{align*}
We choose $\om = (\om_1,\ldots,\om_d)\in\R^d$ and consider the
same~$X_0$ as in~\eqref{eqdefXzeromjxjyj}.
\emph{Theorem~\ref{thmA} can be applied to any
$X\in\cL^\R$ of the form $X_0+[\text{order $\ge1$ in $\eps$}]$ so as to produce $Z,Y\in\cL^\R$
such that $\{X_0,Z\}=0$ and $\ex^{\ad_Y} X = X_0+Z$.}

Indeed, $\cL^\R$ and~$\cL^\C$ are complete filtered Lie algebras
(filtered by the order in~$\eps$), and $B\defeq X-X_0$ can be
decomposed into a formally convergent series as follows:
we can write
$B = \sum\limits_{\N^d\times\N^d} b_{k,\ell}(x,y,\eps)\, x^k y^\ell$ with
$b_{k,\ell}(x,y,\eps) \in \gS_0^\R[[\eps]]\ugeq1$, hence
$B = \sum_{n\in\Z^d} B_n $ with
\[
B_n \defeq \sum_{ \substack{ k',\ell',k'',\ell''\in\N^d \,\text{such} \\[.5ex]
\text{that}\, k'+k'' = n + \ell'+\ell'' }}
\frac{(-\I)^{\abs{\ell'+k''}}}{(\sqrt{2})^{\abs{k'+k''+\ell'+\ell''}}}
%
\binom{k'+\ell'}{k'} \binom{k''+\ell''}{k''}
b_{k'+\ell',k''+\ell''} \,
z(x,y)^{k'+k''} w(x,y)^{\ell'+\ell''}
\]
with the same $z_j, w_j$ as in~\eqref{eq:defzjwj}.
This is the result of using $(x,y)\mapsto(z,w)$ as a change of
coordinates; notice that the decomposition 
$B = \sum b_{k,\ell}(x,y,\eps)\, x^k y^\ell$ is not unique, but the
decomposition $B=\sum B_n$ is, and we have
\[
\{ X_0, B_n \} = \lam(n) B_n, \qquad
\lam(n) = \I \, \scal{n}{\om} \in \C
\qquad\text{for each $n \in \cN \defeq \Z^d$}.
\]
Note that each $B_n \in \cL^\C$, but the realness assumption on~$X$
implies that $\gC(B_n)=B_{-n}$ with the same conjugate-linear
involution~$\gC$ as in Section~\ref{secconjlingC}.
Therefore, for any real-valued $A^\bul \in \Alt\RES(\cN)$, we get
alternal moulds $F^\bul,G^\bul \in \C^\UN$ such that
\[
Z = \sum_{r\ge1} \, \sum_{\un\in\cN^r} \,
\frac{1}{r} F^\un \, \Bun,
\qquad
Y = \sum_{r\ge1} \, \sum_{\un\in\cN^r} \,
\frac{1}{r} G^\un \, \Bun
\]
define $Z,Y \in \cL^\R\ugeq1$ with the desired properties
(the realness of $Z$ and~$Y$ follows from the same argument as in Section~\ref{sec:argumreal}).

Note that if $\om$ is strongly non-resonant in the sense of
Remark~\ref{remStrongNR}, then $Z \in \gS_0^\R[[\eps]]$.


\section{Multiphase averaging}   \label{secmultiphas}


\parage
Let $d,N\in\N^*$.
We call ``slow-fast'' a vector field of the form 
\be   \label{eqvfslowfast}
X = \sum_{j=1}^d \big(\om_j + \eps f_j(\ph,I,\eps) \big)
\frac{\pa\,\;}{\pa\ph_j}
+ \sum_{k=1}^N \eps g_k(\ph,I,\eps) \frac{\pa\,\;}{\pa I_k},
\ee
where $\om = (\om_1,\ldots,\om_d)\in\R^d$ is called the frequency
vector,
the idea being that, for $\eps>0$ ``small'', the time evolution of the
variables~$I_k$ will be ``slow'' compared to the ``fast''
variables~$\ph_j$ (at least if $\om\neq0$).
We take $\ph \in \T^d$, where $\T \defeq \R/2\pi\Z$, so the fast variables are angles.
When $d=N$, this includes the case of vector field generated by a
near-integrable Hamiltonian
\be   \label{eqHamslowfast}
X\Ham = \scal{\om}{I} + \eps h(\ph,I,\eps)
\ee
for the symplectic form $\sum_{j=1}^d \dd I_j \wedge\dd \ph_j$, for
which $f_j = \frac{\pa h\,}{\pa I_j}$ and $g_j = -\frac{\pa h\,}{\pa \ph_j}$.

We will deal with formal series in~$\eps$ whose coefficients are trigonometric
polynomials in~$\ph$ with complex-valued coefficients smooth in~$I$.
More precisely, we take
$f_1, \ldots, f_d, g_1, \ldots, g_N$ or~$h$ in the complex associative
algebra~$\gA^\C$ or the real associative algebra~$\gA^\R$ defined by
\be    \label{eq:defgAR}
\gA^\C \defeq \gS[\ex^{\pm\I\,\ph_1}, \ldots,\ex^{\pm\I\,\ph_d}][[\eps]],
\qquad
\gA^\R \defeq \{\, f\in \gA^\C \mid \ov{f(\ov\ph,I,\ov\eps)} = f(\ph,I,\eps) \,\}
\ee
with $\gS \defeq C^\infty(D,\C)$,
where~$D$ is an open subset of~$\R^N$
(or $D=D'\times\T^{N''}$ with~$D'$ open subset of~$\R^{N'}$ and $N'+N''=N$);
in fact, we could as well take for~$\gS$ a linear subspace of
$C^\infty(D,\C)$, as long as it is stable under multiplication and all
the derivations $\frac{\pa\,\;}{\pa I_k}$, \eg one could take the
Schwartz space $\cS(\R^N,\C)$.

Note that~$\gA^\R$ coincides with the set of fixed points of the conjugate-linear
involution~$\gC$ which maps
$\sum b_{n,p}(I) \, \eps^p \, \ex^{\I\scal{n}{\ph}}$
to 
$\sum \ov{b_{n,p}(I)} \, \eps^p \, \ex^{-\I\scal{n}{\ph}}$.


Let $X_0 \defeq \sum \om_j \frac{\pa\,\;}{\pa\ph_j}$ and $X_0\Ham
\defeq \scal{\om}{I}$.
\emph{The formal averaging problem asks for a formal conjugacy between~$X$
and a vector field $X_0+Z$ which commutes with~$X_0$
or, in the Hamiltonian version, for a formal symplectomorphism~$\Phi$
such that $X\Ham\circ\Phi$ Poisson-commutes with $X_0\Ham$.}
The reader is referred to \cite{LM} and \cite{MSa} for the importance
of this problem.


Let us set $\kk\defeq\C$ and consider the complete
filtered Lie algebra~$\cL^\C$ consisting of vector fields whose
components belong to~$\gA^\C$
(with $[\cdot\,,\cdot] = \vf{\cdot\,}\cdot$)
or, in the Hamiltonian case, $\cL^\C=\gA^\C$ itself (with
$[\cdot\,,\cdot] = \{\cdot\,,\cdot\}$, the Poisson bracket),
filtered by the order in~$\eps$ in both cases.
If we impose furthermore that the components of the vector fields or
the Hamiltonian functions belong to~$\gA^\R$, then we get a real Lie
subalgebra~$\cL^\R$.


\parage 
We can apply Theorem~\ref{thmA} to~$\cL^\C$.
Indeed, any slow-fast system as above can be written as a sum of
eigenvectors of $\ad_{X_0}=\vf{X_0}{\cdot\,}$ 
or $\ad_{X_0\Ham}=\{X_0\Ham,\cdot\,\}$,
\[
X = X_0 + \sum_{n\in\cN} B_n
\quad\text{or}\quad 
X\Ham = X_0\Ham + \sum_{n\in\cN} B_n\Ham, 
\]
with $\cN = \Z^d$ corresponding to all possible Fourier modes:
\[
B_n = \ex^{\I\scal{n}{\ph}} \bigg(
\sum_{j=1}^d b_{n,j}\cc1(I,\eps) \frac{\pa\,\;}{\pa\ph_j}
+ \sum_{k=1}^N b_{n,k}\cc2(I,\eps) \frac{\pa\,\;}{\pa I_k}
\bigg), \qquad
B_n\Ham = \ex^{\I\scal{n}{\ph}} b_n(I,\eps),
\]
with certain coefficients $b_{n,j}\cc1, b_{n,k}\cc2, b_n \in \gS[[\eps]]$.
In both cases, the eigenvalue map is
\be   \label{eqevmapInom}
n\in\Z^d \mapsto \lam(n) = \I \, \scal{n}{\om} \in \C.
\ee
%
%
For any choice of $A^\bul \in \Alt\RES(\cN)$, we thus get $Y,Z \in
\cL^\C$ of order~$\ge1$ in~$\eps$ such that
\[
\vf{X_0}{Z}=0
\ens\text{and}\ens
\ex^{\ad_Y} X =X_0+Z,
\quad\text{or}\quad
\{X_0\Ham,Z\}=0
\ens\text{and}\ens
\ex^{\ad_Y} X\Ham =X_0\Ham+Z.
\]

In the first case, as in Section~\ref{secPDNF}, 
\be
\ex^{\ad_Y} X = \Phi\ii_* X
\ee 
where~$\Phi$ is the formal flow at time~$1$ of the formal vector
field~$Y$.
In the second case, as in Section~\ref{secCBNF}, 
\be
\ex^{\ad_Y} X\Ham = X\Ham\circ\Phi
\ee 
where~$\Phi$ is the formal symplectomorphism obtained by flowing at
time~$1$ along the formal Hamiltonian vector field $\{Y,\cdot\,\}$.
In both cases, 
\be
\text{$Z$ only contains Fourier modes $n\in\cN$ such that
  $\scal{n}{\om}=0$.}
\ee
Therefore, when~$\om$ is strongly non-resonant in the sense of
Remark~\ref{remStrongNR}, the components of the formal vector
field~$Z$ (in the first case) or the formal series~$Z$ (in the second
case) do not depend on~$\ph$, they are formal series in~$\eps$ with
coefficients depending on~$I$ only:
\emph{the formal change of coordinates~$\Phi\ii$ has eliminated the fast
phase~$\ph$ from the vector field.}


If the coefficients $f_1, \ldots, f_d, g_1, \ldots, g_N$ or~$h$ belong
to~$\gA^\R$, \ie if we start with~$X$ or~$X\Ham$ in~$\cL^\R$,
and we take~$A^\bul$ real-valued,
then one gets $Y,Z \in \cL^\R$ for the
same reason as in Section~\ref{secCBNF}:
$\cL^\R$ consists of the fixed points of~$\gC$ which is a real Lie
algebra automorphism\footnote{%
To see it, first observe that $C\col (\ph,I) \mapsto (-\ph,I)$ is
conformal-symplectic with a factor~$-1$
hence the composition with~$C$ is a complex Lie algebra anti-automorphism~$\Th_C$
of~$\cL^\C$, then note that $\gC = \Th_C\circ\SD$ where~$\SD$ is the symmetry associated
with the direct sum 
$\gA^\C = \gR \oplus \I\gR$,
with $\gR \defeq C^\infty(D,\R)[\ex^{\pm\I\,\ph_1}, \ldots,\ex^{\pm\I\,\ph_d}][[\eps]]$
real linear subspace,
and~$\SD$ is a real Lie algebra anti-automorphism because the Lie
bracket of vector fields with components in~$\gR$ has its components
in~$\I\gR$ and, for Hamiltonians, $\{\gR,\gR\} \subset \I\gR$.
}
mapping~$B_n$ to~$B_{-n}$ and is conjugate-linear,
and the complex conjugate of $F^{n_1\cdots n_r}$ is
$F^{(-n_1)\cdots (-n_r)}$ and similarly for~$G^\bul$
(the condition described in footnote~\ref{footA} is sufficient
for this).


\parage  
%
%
\hspace{-.3em}\emph{Remark.}\label{varBNFcl}
In the real Hamiltonian case, $X_0\Ham+Z$ can be considered as a Birkhoff
normal form for $X\Ham = X_0\Ham + \eps h(\ph,I,\eps)$.
If we choose $\gS = \cS(\R^N,\C)$ in~\eqref{eq:defgAR}, then we get
the action-angle analogue of Section~\ref{sec:BNFclassmixte}.




\section{Quantum Birkhoff normal forms}   \label{secQBNF}


\parage
Let~$\cH$ be a complex Hilbert space, with inner product denoted by $\inn{\cdot}{\cdot}$.
In this section, by ``operator'', we mean an unbounded linear operator
with dense domain.

Let us consider an operator~$X_0$ of~$\cH$ which is diagonal
in an orthonormal basis $\ebas = (e_k)_{k\in I}$ of~$\cH$:
\[
X_0 \, e_k = E_k \, e_k, \qquad k\in I,
\]
with eigenvalues $E_k \in \C$, \ie $X_0$ is a normal operator, or 
$E_k \in \R$, in which case $X_0$ is self-adjoint.
Let~$\ACE$ consist of all operators of~$\cH$ whose domain is the dense subspace
$\Span_\C(\ebas)$ and which preserve $\Span_\C(\ebas)$.
Let~$\LRE$ consist of all symmetric operators among the previous ones.
In particular, the restriction of~$X_0$ to $\Span_\C(\ebas)$ belongs
to~$\ACE$, and even to~$\LRE$ in the self-adjoint case.

Notice that an element~$B$ of~$\ACE$ is determined by a complex
``infinite matrix'' $(\bet_{k,\ell})_{k,\ell\in I}$:
\be   \label{eqdefBbetkell}
B e_k = \sum_{\ell\in I}\bet_{k,\ell} \, e_\ell, \qquad
k \in I,
\ee
with the following ``finite-column'' property: 
if $\bet_{k,\ell}\neq0$ then~$\ell$ belongs to
a finite subset of~$I$ depending on~$k$ and~$B$.
The domain of the adjoint operator~$B^*$ then contains $\Span_\C(\ebas)$, and
\[
B^* e_k = \sum_{\ell\in I} \ov{\bet_{\ell,k}} \, e_\ell, \qquad
k \in I.
\]


\begin{lemma}   \label{lemOpgAcLe}
(i)
For $A,B\in \ACE$, there is a well-defined composite operator
$AB\in\ACE$, and for this product $\ACE$ is an associative
algebra over~$\C$.
%

\noindent
(ii)
Let $\hb>0$ be fixed. The formula
\[
\qu{A}{B} \defeq \tfrac{1}{\I\hb} (AB-BA),
\qquad A, B \in \ACE,
\]
makes~$\ACE$ a Lie algebra over~$\C$, which we denote by $\LCE$.
%

\noindent
(iii) 
$\LRE$ is a real Lie subalgebra of $\LCE$,
coinciding with the set of the fixed points of the
involution
\[
\gC \col B\in \ACE \mapsto {B^*}_{| \Span_\C(\ebas)} \in \ACE,
\]
which is a conjugate-linear anti-homomorphism of the associative
algebra~$\ACE$, and a real Lie algebra automorphism of~$\LCE$.
\end{lemma}


\begin{proof}
Obvious.
\end{proof}


\parage
We want to perturb~$X_0$ in~$\LCE$, \resp in~$\LRE$, by a ``small'' perturbation and
work formally, as in a Rayleigh-Schr\"odinger-like situation.
So, we introduce an indeterminate~$\eps$ and consider
\[
\cL^\C \defeq \LCE[[\eps]], 
\quad \text{\resp}\ens
\cL^\R \defeq \LRE[[\eps]], 
\]
as a complete filtered Lie algebra over~$\C$, \resp over~$\R$, filtered by order in~$\eps$.

To decompose an arbitrary perturbation as a sum of eigenvectors of
$\ad_{X_0}$, we notice that, for $B\in\cL^\C$ with matrix
$\big(\bet_{k,\ell}(\eps)\big)_{k,\ell\in I}$ so
that~\eqref{eqdefBbetkell} holds
(with formal series $\bet_{k,\ell}(\eps) \in \C[[\eps]]$),
we can write
\be   \label{eqdecompBLCE}
B = \sum_{(k,\ell)\in I\times I} \ti B_{k,\ell}
\quad\text{with}
\ens
\ti B_{k,\ell} \defeq \ket{e_\ell} \, \bet_{k,\ell}(\eps) \, \bra{e_k} 
%
%
\ee
(here we used the
Dirac notation \ie 
$\ti B_{k,\ell} e_j = \bet_{k,\ell}(\eps) \, e_\ell$ if $j=k$,
$\ti B_{k,\ell} e_j = 0$ else).
The sum in~\eqref{eqdecompBLCE} may be infinite, but it is
well-defined because its action in $\Span_\C(\ebas)$ is finitary.
%
%
%
%
%
One then easily checks that 
\[
\qu{X_0}{\ti B_{k,\ell}} = \tfrac{1}{\I\hb}(E_\ell-E_k) \ti B_{k,\ell}.
\]
%
%
We thus have $B = \sum_{\lam\in\cN} B_n$ with
\be    \label{eqdefBlamQU}
\cN \defeq \big\{\, \tfrac{1}{\I\hb}(E_\ell-E_k) \mid 
(k,\ell) \in I\times I \,\big\}, \qquad
B_\lam \defeq \sum_{ \substack{(k,\ell) \,\text{such that}\\[.5ex]
E_\ell-E_k = \I\hb\lam} } \ket{e_\ell} \, \bet_{k,\ell}(\eps) \, \bra{e_k} 
\quad\text{for $\lam\in\cN$.}
\ee
Note that, if $X_0,B \in \cL^\R$, then 
\be   \label{eqcondreellamBnst}
%
%
\gC(B_\lam) = B_{-\lam}
\quad \text{for any $\lam\in\cN$.}
\ee
%


We thus suppose that we are given a perturbation
$B \in \cL^\C\ugeq1$.
We can apply Theorem~\ref{thmA} to $X = X_0+B\in\cL^\C$, with $\kk=\C$.
For each choice of $A^\bul \in \Alt_0(\cN)$, we get $Z,Y \in
\cL^\C$ of order~$\ge1$ in~$\eps$ such that
\be   \label{eqquXZY}
\qu{X_0}{Z}=0,
\qquad
\ex^{\ad_Y} X =X_0+Z.
\ee

Since $\ACE[[\eps]]$ is a complete filtered associative algebra
and~$Y$ is of order~$\ge1$ in~$\eps$, we can define
$U \defeq \ex^{\frac{1}{\I\hb}Y}$ by the exponential series: it is an
automorphism of $\Span_\C(\ebas)$ formal in~$\eps$, with inverse 
$U\ii = \ex^{-\frac{1}{\I\hb}Y}$,
and $\ex^{\ad_Y} X = U X U\ii$.
So, the second part of~\eqref{eqquXZY} says that
\[
U (X_0+B) U\ii = X_0+Z, 
\qquad U = \ex^{\frac{1}{\I\hb}Y}.
\]
Mould calculus shows that
\[
\tfrac{1}{\I\hb} Y = 
\sum_{r\ge1}\sum_{\lam_1,\ldots,\lam_r\in\cN}
\, (\tfrac{1}{\I\hb})^r
G^{\lam_1\cdots \lam_r}
B_{\lam_r} \cdots B_{\lam_1},
\qquad
U = \ID + 
\sum_{r\ge1} \; \sum_{\lam_1,\ldots,\lam_r\in\cN}\,
(\tfrac{1}{\I\hb})^r
S^{\lam_1\cdots \lam_r}
B_{\lam_r} \cdots B_{\lam_1},
\]
and there is a similar formula for~$U\ii$ involving the mould~$\inv S^\bul$.

If we assume that each eigenvalue~$E_k$ of~$X_0$ is simple (an
assumption analogous to the strong non-resonance condition of
Remark~\ref{remStrongNR}), then it is easy to check that the first
part of~\eqref{eqquXZY} says that~$Z$ is diagonal in the basis~$\ebas$.
In general, it says that~$Z$ is block-diagonal, where the blocks refer
to the partition $I=\bigsqcup I_a$, $I_a \defeq \{\, k\in I \mid E_k =
a \,\}$.

Suppose now that $X_0\in\cL^\R$, \ie it is a self-adjoint operator,
and also $B\in\cL^\R$. Then, in view of~\eqref{eqcondreellamBnst}, by
the same arguments as in Section~\ref{secCBNF} or~\ref{secmultiphas},
we get $Z,Y \in\cL^\R$.
Note that $U$ is then a ``formal unitary operator''.
The formally conjugate operator $X_0+Z$ is called a quantum Birkhoff normal
form for $X_0+B$.


\parage   \label{sec:rm13}
  The simplest example is that of the self-adjoint operator
  $X_0 = -\I\hb \sum \om_j \frac{\pa\,\;}{\pa\ph_j}$ of
  $\cH = L^2(\T^d)$,
which is diagonal in the Fourier basis.
We have $I=\Z^d$ and, for each $k\in\Z^d$,
$e_k = (2\pi)^{-d/2} \, \ex^{\I\,\scal{k}{\ph}}$
and the corresponding eigenvalue is $E_k = \hb \, \scal{k}{\om}$ for $k\in\Z^d$.
In particular,
\[
\tfrac{1}{\I\hb}(E_\ell-E_k) = \I\, \scal{k-\ell}{\om}.
\]

The simplest example for $\cH = L^2(\R^d)$ is the quantum harmonic oscillator
\be   \label{eq:quharmosc}
X_0 = -\dem \hb^2 \De + \sum_{j=1}^d \dem\om_j^2 x_j^2
\ee
(with $\om_1,\ldots,\om_d>0$ given),
for which the spectrum is natually indexed by $I = \N^d$:
\be    \label{eq:evquharmosc}
E_k = \hb \, \scal{k+(\dem,\ldots,\dem)}{\om},
\qquad k \in \N^d,
\ee
and~$\ebas$ is given by the Hermite functions.

In these cases, one can index the eigenvector decomposition $B=\sum B_n$ of
finite-column operators by $\cN=\Z^d$, by a slight modification
of~\eqref{eqdefBlamQU}:
\[
B_n \defeq \sum_{\substack{ (k,\ell)\in\N^d\times\N^d \\ k-\ell =n}} 
\ket{e_\ell} \, \bet_{k,\ell}(\eps)\, \bra{e_k},
\qquad n\in\Z^d.
\]
This way, the eigenvalue map is $\lam(n) = \I \, \scal{n}{\om}$.

Moreover, in these cases, one may wish to restrict oneself to the
``finite-band'' case defined by replacing~$\LRE$ with its
subspace~$\LREfb$ consisting of those elements associated with infinite
matrices $(\bet_{k,\ell})_{k,\ell\in I}$ for which there
exists~$K\in\N$ such that $\bet_{k,\ell}=0$ for $\abs{k-\ell}<K$.
Since $\LRfb \defeq \LREfb[[\eps]]$ is a Lie subalgebra of~$\cL^\R$, we
get $Z,Y\in\LRfb$ whenever we start with a perturbation $B\in\LRfb$ or
order~$\ge1$ in~$\eps$.


\section{Semi-classical limit}\label{secSemiCl}

%
\parage   \label{sec:presentsemicl}
In general the dependence of the eigenvalues $E_k$ in the Planck
constant $\hb$ is very complicated, very often intractable. This makes
the set $\cN=\cN(\hb)$ in \eqref{eqdefBlamQU} very difficult to follow
as $\hb\to 0$. Nevertheless, this difficulty is absent in the two
examples of~$X_0$ of Section~\ref{sec:rm13}, since we have seen that
in these cases we can choose $\cN=\Z^d$ and
$\lam(n) = \I \, \scal{n}{\om}$, thus independent of $\hb$. 
%


We will now consider an operator $X = X_0 + B\QU$ obtained by Weyl
quantization\footnote{See \eg \cite{Fo} for a
  general exposition of pseudo-differential operators and Weyl
  quantization. The few definitions and facts we need will be recalled
in Section~\ref{secWeylQu}.} 
from a classical Hamiltonian $\sig(x,\xi,\eps)$ of the type
introduced in Section~\ref{sec:BNFclassmixte}.
For the sake of simplicity, we choose~$X_0$ to be the quantum harmonic
oscillator~\eqref{eq:quharmosc} on $L^2(\R^d)$
(we could treat as well the case of the trickier Weyl quantization on~$\T^d$ and
choose for~$X_0$ the first example of Section~\ref{sec:rm13}, starting
from a classical Hamiltonian $\sig(x,\xi,\eps)$ of the type alluded to in Section~\ref{varBNFcl}).
We take arbitrary $\om_1,\ldots,\om_d>0$; it will \emph{not} be necessary to
assume that the corresponding frequency vector
$\om \defeq (\om_1,\ldots,\om_d)$ is non-resonant.

The quantum harmonic oscillator~$X_0$ is the Weyl quantization of the Hamiltonian
\be   \label{eq:sigzquadr}
\sig_0(x,\xi) \defeq  \sum_{j=1}^d \dem (\xi_j^2 + \om_j^2 x_j^2),
\ee
which differs from the quadratic Hamiltonian~\eqref{eqdefXzeromjxjyj}
considered in Section~\ref{secCBNF} only by the conformal-symplectic
change of coordinates induced by $\xi_j =\om_j y_j$.
Let us thus consider a formal Hamiltonian $\sig \in \gS^\R[[\eps]]$ of
the form
\be   \label{eq:sigsigzBcl}
\sig = \sig_0 + B\cl, \quad
\text{with $B\cl = B\cl(x,\xi,\eps)$ of order $\ge1$ in~$\eps$,}
\ee
exactly as in Section~\ref{sec:BNFclassmixte} except for the change $y
\to\xi$.
Weyl quantization gives rise to a self-adjoint operator $X = X_0 +
B\QU$ of $L^2(\R^d)$.
We are interested in comparing the quantum Birkhoff normal form $X_0+Z\QU$ of~$X$ and the classical Birkhoff
normal form $\sig_0+Z\cl$ of~$\sig$.

We will see how transparent mould calculus makes the relation between
$Z\QU$ and~$Z\cl$.
The point is that it is the very same mould~$F^\bul$ which will appear
in the mould expansions $Z\cl = F^\bul \Bul\cl$ and $Z\QU = F^\bul
\Bul\QU$;
the difference lies only in the Lie comould to be used in each
expansion, but the semi-classical limit of the
quantum Lie comould $\Bul\QU$ is easily tractable in this context, with
its symbol tending to $\Bul\cl$ as $\hb\to0$.
In fact, all the ``difficult'' part, that is solving the mould equation
which generates combinatorial difficulties solved only by induction,
is exactly the same in the classical and quantum cases.

\parage   \label{secWeylQu}
The operator~$X_0$ is obtained from~$\sig_0$ by
replacing~$\xi_j$ by $-i\hb\frac{\pa\,\;}{\pa x_j}$. 
More generally, Weyl quantization associates to a function~$\sig$
belonging \eg to the Schwartz class
$\cS(\R^d\times\R^d)=\cS(T^*\R^{d})$ an operator~$\V$
which acts on a function $\varphi\in L^2(\R^d)$ through the formula
\be\label{weylq} 
\V\varphi(x)=\int_{\R^d\times\R^d} 
\sig\Big(\frac{x+y}2,\xi\Big) \,\ex^{\I\frac{\xi(x-y)}\hb}\varphi(y)
\frac{\dd\xi \dd y}{(2\pi\hb)^d}.  
\ee 
In other words, the operator~$\V$ has an integral kernel given by
 \[
K_\V(x,y) \defeq \int_{\R^d\times \R^d} \sig\Big(\frac{x+y}2,\xi\Big) \, \ex^{\I\frac{\xi(x-y)}\hb}\frac{\dd\xi }{(2\pi\hb)^d}.
\]
A straightforward computation shows  that this formula is invertible by 
\be\label{symbweylq}
\sig(x,\xi)=\int_{\R^d} K_\V(x+\delta,x-\delta) \,\ex^{-2\I\frac{\xi\delta}\hb}\dd\delta.
\ee
In that situation, we use the notation $\sig = \sig_\V$ and say that the
function~$\sig$ is the ``symbol'' of the operator~$\V$.
For instance, with the notations of Section~\ref{sec:presentsemicl}, $\sig_0 = \sig_{X_0}$.

The following result is the fundamental one concerning the transition
quantum-classical. Its proof is straightforward for symbols in the
Schwartz class, by using~\eqref{weylq} and~\eqref{symbweylq}. It gives
a mod($\hbar)$-homomorphism between quantum and classical Lie
algebras.

\begin{lemma}\label{poissq}
Suppose that the operators~$V$ and~$W$ are obtained by Weyl
quantization from the symbols~$\sig_V$ and~$\sig_W$. Then the symbol
of $\frac{1}{\I\hb}[W,V]$ is
\be\label{adadad}
\sigma_{\frac{1}{\I\hb}[W,V]} = A(\sigma_W\otimes\sigma_V), 
\ee 
where $A(f\otimes g)(x,\xi)=\frac1\hb{\sin{\left(\hb(\frac\partial{\partial
        q}\frac\partial{\partial p'}-\frac\partial{\partial
        p}\frac\partial{\partial
        q'})\right)}f(q,p)g(q',p')}|_{q=q'=x,\, p=p'=\xi}$.

In particular
\be\label{ytr}
\lim_{\hbar\to 0}\sigma_{\frac{1}{\I\hb}[W,V]}=\{\sigma_W,\sigma_V\}
\ee
and, in the case of a quadratic symbol $\sig_{X_0}$ like in~\eqref{eq:sigzquadr},
\be\label{kjh}
\sigma_{\frac{1}{\I\hb}[X_0,V]}=\{\sigma_{X_0},\sigma_V\}.
\ee
\end{lemma}

\parage
On the one hand, according to Section~\ref{sec:BNFclassmixte}, the
Hamiltonian~\eqref{eq:sigsigzBcl} can be decomposed as
\[
\sig =\sig_0 + \sum_{n\in\cN} B\cl_n, \qquad
\{ \sig_0, B\cl_n \} = \lam(n) B\cl_n, \qquad
\lam(n) = \I \, \scal{n}{\om},
\]
with $\cN \defeq \Z^d$.
Denoting by~$\Bul\cl$ the Lie comould defined from
$(B_n\cl)_{n\in\cN}$ by means of Poisson brackets, we get a Birkhoff
normal form of~$\sig$ in the form $\sig_0+Z\cl$ with
\be   \label{eq:BNFcl}
Z\cl = \sum_{r\ge1} \, \sum_{\un\in\cN^r} \,
\frac{1}{r} F^{\lam(n_1)\cdots\lam(n_r)} \, \Bun\cl,
\ee
where we choose for~$F^\bul$ the first of a pair of alternal moulds
$(F^\bul,G^\bul)$ solving~\eqref{eqmouldFG} in the canonical case of
Section~\ref{paragaCanCase}
(we may choose any alternal solution, \eg the zero gauge solution;
note that if~$\om$ is strongly non-resonant, then~$Z\cl$ is uniquely
determined, hence this choice is not relevant, but we make no such
hypothesis about~$\om$).

On the other hand, the Weyl quantization of $\sig = \sig_0 + B\cl$ is
$X=X_0+B\QU$ and, for each $n\in\cN$, the Weyl quantization~$B_n\QU$
of~$B_n\cl$ satisfies
\[
\sig_{\frac{1}{\I\hb}[X_0,B_n\QU]} = \{ \sig_0, B_n\cl \} = \sig_{\lam(n)B_n\QU}
\]
because of~\eqref{kjh}, hence~$B_n\QU$ is the $n$-homogeneous
component of~$B\QU$.
Note that $B\QU$ and the $B_n\QU$'s belong to the space
$\LREfb[[\eps]]$ defined at the end of Section~\ref{sec:rm13}.
Now, according to Section~\ref{secQBNF}, we obtain a quantum Birkhoff normal
form of~$X$ in the form $X_0+Z\QU$ with
\be   \label{eq:BNFQU}
Z\QU = \sum_{r\ge1} \, \sum_{\un\in\cN^r} \,
\frac{1}{r} F^{\lam(n_1)\cdots\lam(n_r)} \, \Bun\QU,
\ee
if we take for~$F^\bul$ the \emph{same} mould as
in~\eqref{eq:BNFcl} and define~$\Bul\QU$ as the Lie comould generated
by $(B_n\QU)_{n\in\cN}$ by means of the Lie bracket
$\qu{\cdot\,}\cdot$ of $\LREfb[[\eps]]$
(note that, \emph{if} $\om$ is strongly non-resonant, then the
eigenvalues~\eqref{eq:evquharmosc} are simple and~$Z\QU$ is uniquely determined).

For each letter $n\in\cN$, the symbol of $B_n\QU$ is the Hamiltonian
$B_n\cl$, but in general, for a word $\un\in\UN$ of length $\ge2$, the
symbol of $\Bun\QU$ is not exactly $\Bun\cl$. However, iteration of~\eqref{ytr}
implies
\be   \label{eq:limsemiclcomld}
\lim_{\hb\to0} \sig_{\Bun\QU} = \Bun\cl
\quad\text{for each nonempty $\un\in\UN$}.
\ee
Putting together \eqref{eq:BNFcl}, \eqref{eq:BNFQU} and~\eqref{eq:limsemiclcomld},
we thus obtain very simply the following result:
 
\begin{THM}
One has
\[
\sig_{Z\QU} \xrightarrow[\hb\to0]{} Z\cl \quad\text{termwise in~$\eps$,}
\]
\ie the coefficients of the $\eps$-expansion of the classical Birkhoff
normal form $X_0+Z\cl$ are the limits, as $\hb\to0$, of the symbols of
the coefficients of the $\eps$-expansion of the quantum Birkhoff
normal form $X_0+Z\QU$.
  \end{THM}

  In the case of a strongly non-resonant frequency vector~$\om$
  satisfying a Diophantine condition, this result was first
  established in \cite{GP1} and later using the Lie method in
  \cite{DGH}.

\vspace{1.5cm}

\noindent \textbf{Acknowledgments}: 
The authors are grateful to Fr\'ederic Menous for pointing out the reference \cite{menousjmp}.
This work has been partially
 carried out thanks to the support of the A*MIDEX project (n$^{\text{o}}$
 ANR-11-IDEX-0001-02) funded by the ``Investissements d'Avenir" French
 Government program, managed by the French National Research Agency
 (ANR). 
 The research leading these results was also partially supported by the
 French National Research Agency under the reference ANR-12-BS01-0017.
T.P.\ thanks the Dipartimento di Matematica, Sapienza
 Universit\`a di Roma, for its kind hospitality during the completion
 of this work.  

\vspace{.5cm}



\begin{flushleft}

\end{flushleft}

\end{document}

%% file: macroQMC.tex
\usepackage{amsmath}
\usepackage{array}
\usepackage{mathabx}
\usepackage{enumerate}
\usepackage{mathrsfs}
\usepackage{dsfont}
\usepackage{amssymb}
\usepackage{graphics}
\usepackage{shuffle}
\usepackage{array}
\newcolumntype{M}[1]{>{\centering\arraybackslash}m{#1}}

\input{amssym.def}
\input{amssym}

\makeatletter
\def\section{\@startsection{section}{1}%
  \z@{1.2\linespacing\@plus\linespacing}{.9\linespacing}%
  {\normalfont\bfseries\large}}
\makeatother

\newtheorem{theorem}{Theorem}

\newtheorem*{THM}{Theorem}
\newtheorem*{addend}{Addendum to Theorem~\ref{thmA}}

\newtheorem{proposition}{Proposition}[section]

\newtheorem{lemma}[proposition]{Lemma}

\theoremstyle{definition}

\newtheorem{definition}[proposition]{Definition}

\newtheorem{remark}[proposition]{Remark}

\makeatletter

\@addtoreset{equation}{section}
\makeatother

\newcounter{parage}[section]
\renewcommand{\theparage}{{\thesection.\arabic{parage}}}
\newcommand{\parage}{\medskip \refstepcounter{parage}
\noindent{\bf\theparage\ } }

\newcounter{parag}[subsection]
\renewcommand{\theparag}{{\thesubsection.\arabic{parag}}}
\newcommand{\parag}{\medskip
  \refstepcounter{parag} \noindent{{\it\theparag.}\hspace{-.2em} }}

\newcounter{paraga}[subsection]
\renewcommand{\theparaga}{{\theparage.\arabic{paraga}}}
\newcommand{\paraga}{\medskip
  \refstepcounter{paraga} \noindent{{\it\theparaga.}\hspace{-.2em} }}

\newcommand{\cL}{{\mathcal L}}
\newcommand{\cF}{{\mathcal F}}
\newcommand{\cM}{{\mathcal M}}
\newcommand{\cN}{{\mathcal N}}
\newcommand{\D}{{\mathcal D}}
\newcommand{\cH}{\mathcal{H}}
\newcommand{\cS}{{\mathcal S}}
\newcommand{\V}{{\mathcal V}}

\newcommand{\vf}{\varphi}

\newcommand{\A}{{\mathcal A}}

\newcommand{\be}{\begin{equation}}
\newcommand{\ee}{\end{equation}}
\newcommand{\bea}{\begin{eqnarray}}
\newcommand{\eea}{\end{eqnarray}}

\newcommand{\om}{\omega}
\newcommand{\Om}{\Omega}

\newcommand{\ad}{\operatorname{ad}}

\newcommand{\lable}{\label}

\newcommand{\N}{\mathbb N}

\newcommand{\R}{\mathbb R}
\newcommand{\T}{\mathbb T}
\newcommand{\Z}{\mathbb Z}
\newcommand{\C}{\mathbb C}
\newcommand{\Q}{\mathbb Q}

\newcommand{\kk}{\mathbf{k}}

\newcommand{\De}{\Delta}
\newcommand{\lam}{\lambda}

\newcommand{\sig}{\sigma}
\newcommand{\Th}{\Theta}
\newcommand{\ph}{\varphi}

\newcommand{\un}{{\underline n}}
\newcommand{\ula}{{\underline\lam}}
\newcommand{\ua}{{\underline a}}
\newcommand{\ub}{{\underline b}}
\newcommand{\UM}{{\underline\cM}}
\newcommand{\UN}{{\underline\cN}}

\newcommand{\est}{{\scriptstyle\diameter}}
\newcommand{\bul}{\bullet}
\newcommand{\bbul}{{\bul,\bul}}
\newcommand{\na}{\nabla}

\newcommand{\I}{{\mathrm i}}
\newcommand{\dd}{{\mathrm d}}
\newcommand{\ex}{{\mathrm e}}

\newcommand{\ID}{\mathop{\hbox{{\rm Id}}}\nolimits}
\newcommand{\inv}{ { {}^{\mathrm{inv}} \! } }
\newcommand{\ti}{\tilde}

\newcommand{\sh}[3]{\operatorname{sh}\!\big( \begin{smallmatrix}#1,\,
#2\\#3\end{smallmatrix} \big)}
\newcommand{\shabn}{\sh{\ua}{\ub}{\un}}
\newcommand{\SH}[3]{\operatorname{sh}\!\Big( \begin{smallmatrix}#1,\,
#2\\[1.5ex]#3\end{smallmatrix} \Big)}

\newcommand{\ens}{\enspace}
\newcommand{\col}{\colon\thinspace}               
\makeatletter
\newcommand*{\defeq}{\mathrel{\rlap{%
                     \raisebox{0.3ex}{$\m@th\cdot$}}%
                     \raisebox{-0.3ex}{$\m@th\cdot$}}%
                     =}
\makeatother                                                   

\newcommand{\ie}{{{i.e.}}\ }
\newcommand{\eg}{{{e.g.}}\ }
\newcommand{\resp}{{resp.}\ }

\newcommand{\lhs}{{left-hand side}}
\newcommand{\rhs}{{right-hand side}}
\newcommand{\End}{\operatorname{End}}
\newcommand{\ord}{\operatorname{ord}}

\newcommand{\ii}{^{-1}}
\newcommand{\pa}{\partial}
\newcommand{\ov}{\overline}
\newcommand{\wt}{\widetilde}

\newcommand{\tr}[2]{{#1}_{#2}}

\newcommand{\eps}{\varepsilon}

\newcommand{\abs}[1]{\lvert #1 \rvert}

\newcommand{\scal}[2]{\langle #1,#2 \rangle}
\newcommand{\inn}[2]{\langle #1 \,|\, #2 \rangle}
\newcommand{\bra}[1]{\langle #1 \,|}
\newcommand{\ket}[1]{|\, #1 \rangle}

\makeatletter
\DeclareRobustCommand{\vf}{\@ifstar\@vf\@@vf}
\newcommand{\@vf}[2]{[ #1, #2 ]_{\mathrm{vf}}}
\newcommand{\@@vf}[2]{\left[ #1, #2 \right]_{\mathrm{vf}}}
\DeclareRobustCommand{\ham}{\@ifstar\@ham\@@ham}
\newcommand{\@ham}[2]{[ #1, #2 ]_{\mathrm{ham}}}
\newcommand{\@@ham}[2]{\left[ #1, #2 \right]_{\mathrm{ham}}}
\DeclareRobustCommand{\qu}{\@ifstar\@qu\@@qu}
\newcommand{\@qu}[2]{[ #1, #2 ]_{\mathrm{qu}}}
\newcommand{\@@qu}[2]{\left[ #1, #2 \right]_{\mathrm{qu}}}
\makeatother

\newcommand{\bet}{\beta}
\newcommand{\Ham}{^{\mathrm{ham}}}
\newcommand{\cl}{^{\mathrm{cl}}}
\newcommand{\QU}{^{\mathrm{qu}}}
\newcommand{\RES}{_{\lam=0}}
\newcommand{\RRES}{_{\mu=0}}
\newcommand{\lres}[1]{\left[ #1 \right]_{\lam=0}}
\newcommand{\gA}{{\mathscr A}}
\newcommand{\gC}{{\mathscr C}}
\newcommand{\gR}{{\mathscr R}}
\newcommand{\gS}{{\mathscr S}}
\newcommand{\JJ}{{\mathscr J}_\lam}
\newcommand{\JJm}{{\mathscr J}_\mu}
\newcommand{\JJz}{{\mathscr J}}
\newcommand{\imp}{\quad \Rightarrow \quad}

\newcommand{\BCH}{\operatorname{BCH}}
\newcommand{\LIE}{\operatorname{Lie}}
\newcommand{\Span}{\operatorname{Span}}
\newcommand{\dem}{\tfrac{1}{2}}
\newcommand{\IND}[1]{{\mathds 1}_{\{ #1 \}}}
\newcommand{\cc}[1]{^{ [#1] }}
\newcommand{\sst}{\sideset{}{^*}}
\newcommand{\Bun}{B_{[\,\un\,]}}
\newcommand{\Bula}{B_{[\,\ula\,]}}
\newcommand{\Bul}{B_{[\,\bul\,]}}
\newcommand{\kkfUN}{\kk^{(\,\UN\,)}}
\newcommand{\kUN}{\kk\,\UN}
\newcommand{\Altf}{\operatorname{Alt}_f^\bul}
\newcommand{\ebas}{\mathbf{e}}
\newcommand{\SD}{\mathbf{S}}

\newcommand{\ACE}{\gA^\C_\ebas}

\newcommand{\LCE}{\cL^\C_\ebas}
\newcommand{\LRE}{\cL^\R_\ebas}
\newcommand{\LREfb}{\cL^\R_{\ebas,\mathrm{fb}}}

\newcommand{\LRfb}{\cL^\R_{\mathrm{fb}}}

\newcommand{\hb}{{\mathchar'26\mkern-9mu h}}
\renewcommand{\hbar}{\hb}

\newcommand{\Alt}{\operatorname{Alt}^\bul}
\newcommand{\Sym}{\operatorname{Sym}^\bul}

\newcommand{\id}{\operatorname{id}}
\newcommand{\card}{\operatorname{card}}
\newcommand{\ugeq}[1]{_{\ge #1}}
\newcommand{\uleq}[1]{_{\le #1}}
\newcommand{\usup}[1]{_{> #1}}
\newcommand{\uinf}[1]{_{< #1}}
\newcommand{\uim}{\uinf m}
\newcommand{\LhS}{\operatorname{LHS}}
\newcommand{\LHS}[1]{\operatorname{LHS}(#1)}
\newcommand{\LHSp}[1]{\operatorname{LHS}'(#1)}
\newcommand{\um}{{\underline m}}
\newcommand{\fS}{{\mathfrak{S}}}
\newcommand{\fA}{{\mathfrak{A}}}
\newcommand{\fB}{{\mathfrak{B}}}
\newcommand{\Bum}{B_{[\,\um\,]}}
\newcommand{\Tun}{T_{[\,\un\,]}}
\newcommand{\Tul}{T_{[\,\bul\,]}}
\newcommand{\shtiabn}{\sh{\wt\ua}{\ub}{\un}}

\newcommand{\cE}{{\mathcal E}}
\newcommand{\kkukk}{{\kk^{\underline{\kk}}}}
\newcommand{\kUNUN}{{\kk^{\UN\times\UN}}}
\newcommand{\shuff}{\shuffle}


%% file: LieDyn_formal_21avril2016.bbl
\begin{thebibliography}{DGH91}


\bibitem[DGH91]{DGH}	
M. Degli Esposti, S. Graffi, J. Herczynski,  {\it Quantization of the classical Lie algorithm in the Bargmann representation}, Annals of Physics, {\bf 209} 2, (1991) 364--392.


\bibitem[Eca81]{Eca81} J. \'Ecalle, 
%
\textit{Les fonctions r\'esurgentes},
%
Publ. Math. d'Orsay [Vol.~1: 81-05, Vol.~2: 81-06, Vol.~3: 85-05]
%
1981, 1985.

%
%

\bibitem[Eca93]{Eca93} J.~\'Ecalle,
%
Six lectures on Transseries, Analysable Functions and the
Constructive Proof of Dulac's conjecture,
%
in \emph{Bifurcations and periodic orbits of vector fields (Montreal, PQ,
1992)} (ed. by D. Schlomiuk), NATO Adv. Sci. Inst. Ser.C Math. Phys. Sci. 408, Kluwer
Acad. Publ., Dordrecht 1993, 75--184.


%
%

\bibitem[EV95]{EV95}
%
J.~\'Ecalle and B.~Vallet,
%
\emph{Prenormalization, correction, and linearization of resonant vector fields or diffeomorphisms.}
%
Prepub. Orsay 95-32 (1995), 90 pp. 

%



\bibitem[Fol89]{Fo} G. Folland, \emph{Harmonic Analysis in Phase Space}, Annals of Mathematics Studies \textbf{122}, Princeton University
Press 1989.

\bibitem[GP87]{GP1} S. Graffi, T.Paul, {\it Schr\"odinger equation and canonical perturbation theory}, Comm. Math. Phys., {\bf 108} (1987), 25--40.



\bibitem[LM88]{LM} P.~Lochak, C.~Meunier,
%
\emph{Multiphase averaging for classical systems},
%
Applied Mathematical Sciences, {\bf 72},
%
{Springer-Verlag, New York}, {1988}, xii+360~pp.


\bibitem[MS02]{MSa} J.-P.~Marco, D.~Sauzin, ``Stability and instability for Gevrey
quasi-convex near-integrable Hamiltonian systems,'' {\em Publications
Math\'ematiques de l'Institut des Hautes \'Etudes Scientifiques} {\bf 96}
(2002), 199--275.


\bibitem[Men13] {menousjmp} F. Menous, \textit{From dynamical systems
    to renormalization}, 
Journal of Mathematical Physics {\bf 54} (2013), 092702 1-24.



 \bibitem[PS16]{Dyn} T. Paul and D. Sauzin, \textit{Normalization in
     Banach scales of Lie algebras via mould calculus and applications}, in preparation.




\bibitem[Sau09]{mouldSN} D. Sauzin, ``Mould expansions for the
    saddle-node and resurgence monomials,'' 
%
in \emph{Renormalization and Galois theories}, p.~83--163, 
%
A.~Connes, F.~Fauvet, J.-P.~Ramis (eds.),
%
IRMA Lectures in Mathematics and Theoretical
Physics 15, Z\"urich: European Mathematical Society, 2009.
%
    



\end{thebibliography}
